\documentclass[DIV=14,letterpaper]{scrartcl}

\usepackage{tikz}
\usetikzlibrary{calc}
\usetikzlibrary{matrix}
\usepackage{amsmath,amssymb,amsfonts,amsthm}
\usepackage{graphicx}
\allowdisplaybreaks
\usepackage{subfigure,multicol,multirow}
\usepackage{makecell}
\usepackage{hyperref} 
\usepackage{diagbox}

\usepackage{bbm}

\theoremstyle{plain}% default
\newtheorem{thm}{\hspace{6mm}Theorem}[section]
\newtheorem{pro}{\hspace{6mm}Proposition}[section]
\newtheorem{lem}{\hspace{6mm}Lemma}[section]

\theoremstyle{definition}

\theoremstyle{remark}

\newcommand{\V}[1]{\mathbf{#1}}

\newcommand{\email}[1]{\href{mailto:#1}{#1}}

\title{
An inherent regularization approach to parameter-free preconditioning
for nearly incompressible linear poroelasticity 
}
\author{
Weizhang Huang\thanks{Department of Mathematics, the University of Kansas, 1460 Jayhawk Boulevard, Lawrence, KS 66045, USA (\email{whuang@ku.edu}).}
\and
Zhuoran Wang\thanks{Department of Mathematics, the University of Kansas, 1460 Jayhawk Boulevard, Lawrence, KS 66045, USA (\email{wangzr@ku.edu}). Corresponding author.}
}
\date{} 

\begin{document}

\maketitle

\textbf{Abstract.}
An inherent regularization strategy combined with block Schur complement preconditioning is investigated
for linear poroelasticity problems discretized using the lowest-order weak Galerkin method in space
and the implicit Euler scheme in time. At each time step, the resulting saddle point system becomes
nearly singular in the locking regime due to the near incompressibility of the solid.
The regularization arises from an inherent identity in the original system,
preserves the exact solution, and ensures nonsingularity.
For the regularized systems, standard inexact block Schur complement preconditioning
is shown to be effective, and preconditioned MINRES and GMRES methods exhibit convergence
essentially independent of the mesh size and locking parameter -- in terms of both convergence factors
and asymptotic error constants.
% The regularization strategy is also examined for linear elasticity problems.
Numerical experiments in two and three dimensions confirm the robustness of the regularization strategy
and the effectiveness of the block preconditioners with respect to mesh size and locking
for the regularized systems.

\vspace{5pt}

\noindent
\textbf{Keywords:}
% MINRES, GMRES, Preconditioning, Weak Galerkin.
Inherent regularization, Poroelasticity, Block preconditioning, Parameter-free, Weak Galerkin.

\vspace{5pt}

\noindent
\textbf{Mathematics Subject Classification (2020):}
65N30, 65F08, 65F10, 74F10

%%%%%%%%%%%%%%%%%%%%%%%%%%%%%%%%%%%%%%%%%%%%%%%%%%%%%%%%%%%%%%%
%%%%%%%%%%%%%%%%%%%%%%%%%%%%%%%%%%%%%%%%%%%%%%%%%%%%%%%%%%%%%%%
\section{Introduction}
\label{SEC:intro}

We consider an inherent regularization strategy for developing parameter-free
inexact block Schur complement preconditioners for the efficient iterative
solution of linear poroelasticity.
Let $ \Omega \subset \mathbb{R}^d$ $(d\ge 1) $ be a bounded domain with
Lipschitz continuous boundary $\partial \Omega$.
The governing equation for the Biot model of linear poroelasticity is given by
\begin{align}
\begin{cases}
  \displaystyle
  -\nabla \cdot \big ( 2\mu \varepsilon(\V{u}) + \lambda (\nabla \cdot \V{u}) \mathbf{I} \big )
  + \alpha \nabla p
  = \mathbf{f}, \quad \text{ in } \Omega  \times (0,T],
\\
  \displaystyle
  \partial_t
  \big ( \alpha \nabla \cdot \V{u} +  c_0 p \big)
    - \nabla \cdot \left( \kappa \nabla p \right)
  = s, \quad \text{ in } \Omega  \times (0,T],
\end{cases}
\label{PDE-poro}
\end{align}
where
$T > 0$ denotes the final time,
$ \V{u} $ is the solid displacement,
$ \lambda$ and $\mu $ are the Lam\'{e} constants,
% $\sigma(\V{u}) = 2\mu \varepsilon(\V{u}) + \lambda (\nabla \cdot \V{u}) \mathbf{I}$
% is the Cauchy stress for the solid,
$ \varepsilon(\V{u}) = \frac{1}{2} ( \nabla \V{u} + (\nabla \V{u})^T ) $ is the strain tensor, 
$\mathbf{I} $ is the identity operator,
$ \mathbf{f} $ is a body force,
$ p $ is the fluid pressure,
$ s $ is a fluid source,
$ \alpha $ (usually close to $1$) is the Biot-Willis constant accounting for the coupling of the solid and fluid,
$ c_0 > 0 $ is the constrained storage capacity,
and $ \kappa $ is the permeability constant.
We consider
Dirichlet boundary conditions for both the displacement and pressure, i.e., 
\begin{align}
\V{u} = \V{u}_D, \quad p = p_D, \quad \text{on } {\partial \Omega}  \times (0,T],
\label{BC-1}
\end{align}
where $\V{u}_D$ and $p_D$ are given functions.
The initial conditions are specified as
\begin{align}
 \V{u} = \V{u}_0,
  \quad
  p = p_0,
  \quad
  \text{on} \; \Omega \times \{t=0\}.
  \label{IC-1}
\end{align}

A major challenge in the numerical solution of poroelasticity problems arises in the locking regime,
when the solid becomes nearly incompressible as $\lambda \to \infty$, i.e., $\nabla \cdot \V{u} \to 0$.
In this situation, the Poisson phenomenon \cite{Yi_SISC_2017} may occur: numerical schemes that fail
to preserve the correct asymptotic behavior typically exhibit deteriorating convergence as
the mesh is refined. Another difficulty is related to near singularity. As $\lambda \to \infty$,
the divergence term in the first equation of \eqref{PDE-poro} becomes dominant,
and the resulting linear algebraic system becomes nearly singular when pure Dirichlet boundary
conditions are imposed.

There has been extensive work on the development of locking-free numerical methods for poroelasticity.
Locking-free discretization approaches include
mixed finite element methods \cite{AmbartKhatYotov_CMAME_2020,He-Jing-Feng-2025,ZhaoChenCaiSun},
virtual element methods \cite{Burger4_AdvComputMath_2021,Coulet4_ComputGeosci_2020},
discontinuous Galerkin methods \cite{RIVIERE2017666},
enriched Galerkin methods \cite{KADEETHUM2021110030,LeeYi_JSC_2023},
and weak Galerkin (WG) finite element methods \cite{Wang2TavLiu_JCAM_2024}.
At each time step, these discretizations give rise to large-scale linear systems.
When $\lambda$ is large, the resulting systems become nearly singular, making them challenging
to solve efficiently. Effective preconditioning is therefore key to the efficient solution of these systems.
For example, Lee et al. \cite{Lee-SISC-2017} introduced parameter-robust three-field block diagonal
preconditioners based on stability considerations and the operator preconditioning framework.
Adler et al. \cite{Adler6_SISC_2020} proposed norm-equivalent and field-of-value-equivalent
block preconditioners for stabilized discretizations of poroelasticity problems
within a three-field formulation. Chen et al. \cite{ChenHongXuYang_CMAME_2020} developed
a general framework for poroelasticity and discussed several preconditioners
for two- and three-field numerical schemes. Boon et al. \cite{Boon4_SISC_2021} constructed
parameter-robust preconditioners for four-field formulations.
More recently, Hong et al. \cite{Hong-2023-MathComp} proposed a framework
for stability analysis and the construction of norm-equivalent preconditioners
for perturbed saddle point problems. Rodrigo et al. \cite{Rodrigo6_SeMA_2024} presented
preconditioners for two- and three-field schemes, and Luber and Sysala \cite{Luber-2024}
investigated block diagonal preconditioners for a three-field formulation
of the Biot model of poroelasticity.

It is worth noting that most existing preconditioners are equivalent to the underlying system
in terms of spectrum, norm, and/or field of value. While these preconditioners are generally
effective and parameter-robust, they are singular or nearly singular due to their spectral
equivalence to the original system. This property makes them more difficult to construct
and more expensive to invert than nonsingular preconditioners.

The objective of this work is to develop parameter-free inexact block Schur complement preconditioners
for the poroelasticity problem \eqref{PDE-poro}. For spatial discretization, we employ the lowest-order
weak Galerkin (WG) method \cite{WANGYe2013103}, which has been shown in \cite{Wang2TavLiu_JCAM_2024}
to be locking-free and to achieve optimal-order convergence for both pressure and displacement.
For simplicity, we use the implicit Euler scheme for the temporal discretization of \eqref{PDE-poro}.
We consider a three-field formulation of linear poroelasticity, obtained by introducing a numerical
pressure variable into the leading block of the original two-field system. 
Various three-field formulations have been proposed in the literature. The numerical pressure variable introduced in this work, as well as in \cite{HuangWang_2025_CGWG}, arises purely from the requirements of algebraic manipulation and differs slightly from the total pressure defined in \cite{Lee-SISC-2017}, which is based on a continuous weak formulation.
A rank-one regularization term
is added to the resulting three-field formulation; this term arises naturally from an inherent identity
in the original system and preserves the solution. We demonstrate that, for the resulting regularized problem, the preconditioned minimal residual method (MINRES) and
generalized minimal residual method (GMRES) exhibit parameter-free convergence.

It is worth noting that the iterative solution of saddle point problems has been studied extensively;
see, for example, \cite{BenziGolubLiesen-2005,Benzi2008,Boffi2008,Elman-2014,PestanaWathen-2015}
and the references therein. Moreover, in this work we consider Dirichlet boundary conditions
for the displacement in \eqref{PDE-poro}. If Neumann boundary conditions are imposed on part or all of the domain boundary, the resulting algebraic system is no longer nearly singular, even for large values of $\lambda$.
In such cases, block Schur complement preconditioning is known to be effective for linear poroelasticity; see the aforementioned references as well as \cite{SilvesterWathen_SINUM_1994,WathenSilvester_SINUM_1993}.

It is also worth noting that several recent works, including \cite{Huang-Wang-2025-Stokes-reg},
\cite{HuangWang_2025}, and \cite{HuangWang_2025_CGWG}, are relevant to the present study.
The inherent regularization strategy was investigated in \cite{Huang-Wang-2025-Stokes-reg}
for singular Stokes problems. The algebraic systems considered there,
which are generally singular and inconsistent, differ fundamentally from those studied here,
which become nearly singular for large values of $\lambda$ in the context of linear poroelasticity.
Moreover, extending the same regularization strategy to linear poroelasticity is nontrivial.
In particular, the regularization term must be chosen carefully to preserve the solution
of the original system. In addition, careful selection of the regularization parameter
is required to avoid the introduction of small eigenvalues in the preconditioned system.
We emphasize that the primary contribution of this work is the inherent regularization 
strategy for the poroelasticity problem. The preconditioners analyzed here are 
standard block Schur complement approaches applied to the regularized system.

Iterative solvers with inexact block preconditioning for linear poroelasticity were studied in~\cite{HuangWang_2025}.
The main difference between~\cite{HuangWang_2025} and the present work lies in the treatment of regularization. 
In~\cite{HuangWang_2025}, preconditioning strategies and iterative methods were applied and analyzed
without regularizing the system,  whereas in the current study the system is regularized prior to preconditioning.
It was shown in~\cite{HuangWang_2025} that, in the absence of regularization, MINRES and GMRES equipped
with appropriate inexact block preconditioners converge with essentially parameter-independent convergence factors. 
However, the asymptotic error constants still depend on the mesh size and the locking parameter due to the presence of a small eigenvalue in the preconditioned system as $\lambda \to \infty$. 
This parameter dependence implies that MINRES and GMRES typically require several additional iterations
to resolve this small eigenvalue.
In contrast, the present work shows both analytically and numerically that a suitable regularization of the system significantly improves the convergence behavior of preconditioned MINRES and GMRES. 
In particular, both the convergence factors and the asymptotic error constants become independent of the mesh size
and the locking parameter, and the number of iterations required for convergence remains more stable
with respect to variations in the parameters. In this sense, the current work can be viewed as
an improvement over~\cite{HuangWang_2025}; see the comparative numerical results
in Tables~\ref{Poro-3field-2D-v2} and~\ref{Poro-3field-2D-v2-noreg}.

The inherent regularization strategy considered here has also been employed in \cite{HuangWang_2025_CGWG}
for linear poroelasticity, but within a different discretization framework, namely a hybrid Bernardi–Raugel and
weak Galerkin finite element discretization. By contrast, the present work focuses on
a fully weak Galerkin discretization and is carried out in a different discrete setting.

The remainder of the paper is organized as follows. In Section~\ref{SEC:formulation},
we describe the discretization of the poroelasticity problem \eqref{PDE-poro} and
discuss its properties. Section~\ref{sec::reg} is devoted to the three-field formulation
of linear poroelasticity and the inherent regularization strategy.
In Section~\ref{sec:poro3}, we study the iterative solution of the three-field formulation
using MINRES and GMRES with block Schur complement preconditioning.
Numerical results for poroelasticity problems in two and three dimensions are presented in Section~\ref{SEC:numerical}. Finally, conclusions are drawn in Section~\ref{SEC:conclusions}.

% Appendix~\ref{app:A} contains a proof of Lemma~\ref{lem:eigen_bound} that provides
% lower and upper bounds on the eigenvalues of the Schur complement preconditioned by
% an approximate Schur complement for linear elasticity.

% %%%%%%%%%%%%%%%%%%%%%%%%%%%%%%%%%%%%%%%%%%%%%%%%%%%%%%%%%%%%%%%
\section{Weak Galerkin discretization}
\label{SEC:formulation}

In this section we describe the lowest-order WG discretization for poroelasticity problem
\eqref{PDE-poro}.

% We start with the poroelasticity problem \eqref{PDE-poro}.
The weak formulation of \eqref{PDE-poro} is 
to seek $(\V{u}(\cdot, t), p(\cdot, t)) \in (H^1(\Omega))^d\times H^1(\Omega)$, $0< t \le T$, such that
\begin{equation}
\begin{cases}
    \mu \Big( \nabla \V{u}, \nabla\V{v} \Big)
      + (\lambda+\mu) ({\nabla \cdot \V{u}}, {\nabla \cdot \V{v}})
      - \alpha (p, {\nabla \cdot \V{v}})
      = (\mathbf{f}, \V{v}),
      \quad \forall \V{v}\in (H_{0}^1(\Omega))^d,
\\
    \displaystyle
    -\alpha (\nabla \cdot \V{u}_t, q)
    - c_0 \left( p_t, q \right)
      - \left( {\kappa} \nabla p, \nabla q \right)
    \displaystyle
    =  -\left( s, q \right),
    \quad \forall q\in H_{0}^1(\Omega),
\end{cases}
  \label{poro_variationalform}
\end{equation}
where $(\cdot, \cdot)$ denotes the $L^2$ inner product over $\Omega$ and we have used the identity
\[ 
   2 \mu \Big ( \varepsilon (\V{u}), \varepsilon(\V{v}) \Big) + \lambda ({\nabla \cdot \V{u}}, {\nabla \cdot \V{v}})= 
    \mu \Big( \nabla \V{u}, \nabla\V{v} \Big)
      + (\lambda+\mu) ({\nabla \cdot \V{u}}, {\nabla \cdot \V{v}}),
\]
which holds when pure Dirichlet boundary conditions are used.
Assume that a quasi-uniform simplicial mesh $\mathcal{T}_h = \{K\}$  is given for $\Omega$,
where $h$ is the maximum element diameter.
Denote the boundary of element $K$ by $\partial K$.
Define the discrete weak function spaces as
\begin{align}
     \displaystyle
    \V{V}_h
     & = \{ \V{u}_h = \{ \V{u}_h^\circ, \V{u}_h^\partial \}: \;
      \V{u}_h^\circ|_{K} \in (P_0(K))^d, \;
      \V{u}_h^\partial|_e \in (P_0(e))^d, \;
      \forall K \in \mathcal{T}_h, \; e \in \partial K \},
    \\ 
    \displaystyle
    W_h  & = \{p_h=\{p^\circ_h, p^\partial_h\}:
    p^\circ_h|_{K} \in P_0(K),
    p^\partial_h|_e \in P_0(e), 
    \; \forall K \in \mathcal{T}_h,  e \in \partial K \},
    \\
    \displaystyle
     \mathcal{P}_0 &= \Big\{p_h=\{p^\circ_h\}\; :\; p^\circ_h|_K\in P_0(K),\; \forall K\in\mathcal{T}_h\Big\} ,
\end{align}
where $P_0(K)$ and $P_0(e)$ denote the spaces of constant polynomials defined on element $K$ and facet $e$, respectively.
Notice that functions in $\V{V}_h$ and $W_h$ consist of two parts, one defined in the interiors of the mesh elements
and the other on their facets.
Let $\V{V}_h^0$ and $W_h^0$ be the subspaces of  $\V{V}_h$ and $W_h$ with vanishing
Dirichlet boundary conditions for $\V{u}_h$ and $p_h$, respectively.
Define the discrete weak gradient operator $\nabla_w: W_h \rightarrow RT_0(\mathcal{T}_h)$  for $u_h = (u_h^{\circ},u_h^{\partial})$ as
\begin{equation}
\label{weak-grad-1}
  (\nabla_w u_h, \mathbf{w})_K
  = (u^\partial_h, \mathbf{w} \cdot \mathbf{n})_{\partial K}
  - ( u^\circ_h , \nabla \cdot \mathbf{w})_K,
  \quad \forall \mathbf{w} \in RT_0(K),\quad \forall K \in \mathcal{T}_h ,
\end{equation}
where 
$\mathbf{n}$ is the unit outward normal to $\partial K$,
$(\cdot, \cdot)_K$ and $( \cdot, \cdot )_{\partial K}$ are the $L^2$ inner product on $K$ and $\partial K$, respectively,
and $RT_0(K)$ is the lowest-order Raviart-Thomas space on element $K$ defined as
\[
RT_0(K) = (P_0(K))^d + \mathbf{x} \, P_0(K).
\]
The global Raviart-Thomas space over the whole mesh $\mathcal{T}_h$ is defined as
\begin{align*}
    & RT_0(\mathcal{T}_h) = \{ \V{u}_h: \; \V{u}_h|_K \in RT_0(K), \; \forall K \in \mathcal{T}_h \}.
\end{align*}
The analytical expression of $\nabla_w u_h$ can be obtained; see, e.g., \cite{HuangWang_CiCP_2015}.
For a vector-valued function $\V{u}_h$, $\nabla_w \V{u}_h$ is a matrix,
with each row corresponding to the weak gradient of a component.
The discrete weak divergence operator $\nabla_w \cdot: \V{V}_h \to \mathcal{P}_0$ is defined
separately as 
\begin{equation}
   (\nabla_w \cdot \V{u}_h, w )_{K}
  = ( \V{u}_h^\partial , w \mathbf{n})_{\partial K}
  - ( \V{u}_h^\circ , \nabla w)_{K},
  \quad
  \forall w \in P_0(K) .
  \label{wk_div1}
\end{equation}
Notice that $(\nabla_w \cdot \V{u}_h)|_K \in P_0(K)$. 
 
For temporal discretization we consider a time partition of the interval $(0,T]$ given by  $0 = t_0 < t_1< ... <t_N = T$, and denote the time step as $\Delta t_n = t_n - t_{n-1}$. Using the implicit Euler scheme for temporal discretization and
the lowest-order WG for spatial discretization,
we obtain the time marching scheme for \eqref{poro_variationalform} as: seek $\V{u}_h^n \in \V{V}_h$ and $p_h^n \in W_h$ 
such that 
% $\V{u}_h|_{\partial \Omega}= \V{u}_D^h$,
% $ p_h|_{\partial \Omega}=p_D^h $, where $\V{u}_D^h$ and $p_D^h$ are
% the $L^2$-projections of $\V{u}_D$ and $p_D$ on $\partial \Omega$,
\begin{equation}
\begin{cases}
\displaystyle
\mu \sum_{K \in\mathcal{T}_h} (\nabla_w\V{u}_h^n,\nabla_w\V{v}_h)_K
+(\lambda+\mu) \sum_{K \in\mathcal{T}_h}(\nabla_w\cdot\V{u}_h^n,\nabla_w\cdot\V{v}_h)_K
\\
\displaystyle
\qquad \qquad \qquad 
- \alpha \sum_{K \in\mathcal{T}_h} (p_h^{\circ,n}, \nabla_{w}\cdot\V{v}_h)_K
= \sum_{K \in\mathcal{T}_h}(\mathbf{f}^n,\V{v}_h^{\circ})_K, \quad \forall \V{v}_h \in \V{V}_h^0,
\\
\displaystyle
-\alpha \sum_{K \in\mathcal{T}_h}  (\nabla_{w}\cdot\V{u}_h^n,q_h^{\circ})_K
- c_0 \sum_{K \in\mathcal{T}_h} (p_h^{\circ,n}, q_h^{\circ})_K
- \Delta t_n \sum_{K \in\mathcal{T}_h}(\kappa\nabla_w p_h^n,\nabla_wq_h)_K
\\
\displaystyle
\qquad \qquad \qquad 
= - \Delta t_n \sum_{K \in\mathcal{T}_h} (s^n,q_h^{\circ})_K
    -\alpha \sum_{K \in\mathcal{T}_h}  (\nabla_{w}\cdot\V{u}_h^{n-1},q_h^{\circ})_K
\\
\displaystyle
\qquad \qquad \qquad \qquad \qquad \qquad \qquad
- c_0 \sum_{K \in\mathcal{T}_h} (p_h^{\circ,n-1}, q_h^{\circ})_K
    ,\quad \forall q_h \in {W}_h^0 .
\end{cases}
\label{EqnFullDisc1}
\end{equation}
Suppressing the subscript and superscript $n$,
we can rewrite the above system in a matrix-vector form  as 
\begin{equation}
    \begin{bmatrix}
        \mu A_1  + ( \lambda + \mu ) A_0 & -\alpha B^T \\
       - \alpha B & -D
    \end{bmatrix}
    \begin{bmatrix}
        \V{u}_h \\
        \mathbf{p}_h
    \end{bmatrix}
    =
    \begin{bmatrix}
        \mathbf{b}_1 \\
        \mathbf{b}_2
    \end{bmatrix},
    \label{2by2Scheme_matrix}
\end{equation}
where the matrices and right-hand sides are given by
\begin{align}
& \V{v}_h^T A_0 \V{u}_h = \sum_{K \in\mathcal{T}_h} (\nabla_w\cdot\V{u}_h,\nabla_w\cdot\V{v}_h)_K,
\quad \forall \V{u}_h, \V{v}_h \in \V{V}_h^0,
\label{A0-1}
\\
& \V{v}_h^T A_1 \V{u}_h = \sum_{K \in\mathcal{T}_h} (\nabla_w\V{u}_h,\nabla_w\V{v}_h)_K,
\quad \forall \V{u}_h, \V{v}_h \in \V{V}_h^0,
\label{A1-1}
\\
& \mathbf{q}_h^T B \V{u}_h =  \sum_{K \in \mathcal{T}_h}  (\nabla_{w}\cdot\V{u}_h,q_h^{\circ})_K,
\quad \forall \V{u}_h \in \V{V}_h^0, \quad \forall q_h \in W_h^0,
\label{B-1}
\\
& \mathbf{q}_h^T D \mathbf{p}_h = c_0 \sum_{K \in \mathcal{T}_h} (p_h^{\circ}, q_h^{\circ})_K 
+ \Delta t \sum_{K \in\mathcal{T}_h}(\kappa\nabla_w p_h,\nabla_wq_h)_K,
\quad \forall p_h, q_h \in W_h^0 ,
\label{D-1}
\\
& \V{v}_h^T \V{b}_1 = \sum_{K \in\mathcal{T}_h} (\V{f}, \V{v}_h^{\circ})_K, \quad \forall \V{v}_h \in \V{V}_h^0,
\label{b1-1}
\\
& \V{q}_h^T \V{b}_2 = - \Delta t \sum_{K \in\mathcal{T}_h} (s,q_h^{\circ})_K
    -\alpha \sum_{K \in\mathcal{T}_h}  (\nabla_{w}\cdot\V{u}_h,q_h^{\circ})_K
    - c_0  \sum_{K \in\mathcal{T}_h} (p_h^{\circ}, q_h^{\circ})_K, \quad \forall q_h \in W_h^0.
\label{b2-1}
\end{align}
Here, the same notation (for example, $\V{u}_h$)
is used for a discrete function and the vector formed by its degrees of freedom
(with the degrees of freedom in the element interiors first and followed by those
on element facets).
Notice that $\V{b}_1$ and $\V{b}_2$ are vectors whose first block is nonzero
and has a size of $d \times N$ and $N$, respectively, while the second block is zero,
where $N$ is the number of elements of $\mathcal{T}_h$.
That is, they have the structure
\begin{align*}
    \V{b}_1 = 
    \begin{bmatrix}
        \V{b}_1^{\circ}
        \\
        0
    \end{bmatrix},\quad 
    \V{b}_2 = 
    \begin{bmatrix}
        \V{b}_2^{\circ}
        \\
        0
    \end{bmatrix}.
\end{align*}

The numerical scheme \eqref{2by2Scheme_matrix} is known to be locking-free and 
to achieve optimal-order convergence for both the displacement $\V{u}$ and the pressure $p$ (see \cite{Wang2TavLiu_JCAM_2024}). 
The {inf-sup} condition and optimal-order convergence are stated in the following propositions.

\begin{pro}
\label{inf_sup}
There exists a constant $ 0<\beta<1 $ independent of $h$ such that 
\begin{align}
   \sup_{\V{v}_h\in \V{V}_h^0, \quad \V{v}_h^T A_1 \V{v}_h \neq 0}
  \frac{\V{p}_h^T B \V{v}_h}
      {(\V{v}_h^T A_1 \V{v}_h)^{\frac{1}{2}}} 
    \geq \beta \|\V{p}_h\|, 
  \qquad \forall p_h \in W_h.
\end{align}
\end{pro}
\begin{proof}
The proof can be found in \cite[Lemma 1]{Wang2TavLiu_JCAM_2024}.
\end{proof}

\begin{pro}
\label{pro:poro_conv}
Let $ (\V{u}, p) $ and $(\V{u}_h,p_h) $ be the exact and numerical solutions of poroelasticity problem (\ref{PDE-poro}).
Under suitable regularity assumptions for the exact solution,
% \begin{align}
%   \V{u} \in [L^{\infty}(0,T;H^{2}(\Omega))]^2,
%   \quad
%   p \in L^{\infty}(0,T;H^{2}(\Omega)),
%   \\
%   \partial_t\V{u}\in [L^{\infty}(0,T;H^{2}(\Omega))]^2,
%   \quad
%   \partial_tp\in L^{\infty}(0,T;H^{1}(\Omega)),
%   \\
%   \partial_{tt}\V{u}\in [L^2(0,T;H^{2}(\Omega))]^2,
%   \quad
%   \partial_{tt}p\in L^2(0,T;H^{1}(\Omega)),
% \label{pro:poro_conv-1}
% \end{align}
there holds 
\begin{align}
  \displaystyle
  \max_{1\leq n\leq N} \left\{\|\nabla \V{u}^n - \nabla_w \V{u}_h^n \|
   + \| p^n-p_h^n \| \right\}
  \displaystyle
  \leq C_1 h + C_2 \Delta t,
  \label{pro:poro_conv-2}
\end{align}
where $C_1$ and $C_2$ depend on $\V{u}$ and $p$ but not on $h$, $\Delta t$, and $\lambda$.
\end{pro}

\begin{proof}
    The proof can be found in \cite[Theorem 1, Theorem 3]{Wang2TavLiu_JCAM_2024}.
\end{proof}

In this work, we focus on constructing effective and parameter-free inexact block preconditioners
for the iterative solution of poroelasticity problems.
To this end, we first analyze the structures of the blocks in the algebraic system \eqref{2by2Scheme_matrix}. 
The block $A_1$ is the stiffness matrix of the Laplacian operator for the displacement and
is symmetric and positive definite (SPD). 
The block $B$ has the structure 
\begin{equation}
B = \begin{bmatrix} B^{\circ} \\ 0 \end{bmatrix},
\label{B-2}
\end{equation}
where 
$\mathbf{q}_h^T B^{\circ} \V{u}_h =  \sum_{K \in \mathcal{T}_h}  (\nabla_{w}\cdot\V{u}_h,q_h^{\circ})_K$ and
the row number of the zero block is equal to the number of the degrees of freedom of $q_h^{\partial}$.
Moreover, $(B^{\circ})^T$ is a rank-1 deficient matrix as shown in the following lemma.

\begin{lem}
\label{lem:B0-1}
The null space of $(B^{\circ})^T$ is given by
\[
\text{Null}((B^{\circ})^T) = \{ p_h \in W_h: \; p_{h,K} = C, \; \forall K \in \mathcal{T}_h, \; \text{ C is a constant} \} .
\]
\end{lem}

\begin{proof}
The proof can be found in \cite[Lemma 2.1]{HuangWang_CiCP_2025}.
\end{proof}

The block $D$ in \eqref{D-1} has the structure as 
\begin{equation}
D = c_0 \begin{bmatrix}
    M_p^{\circ} & 0 \\[0.1in]
    0 & 0
\end{bmatrix}
+ \kappa \Delta t  A_p .
\label{D-2}
\end{equation}
Here, $A_p$ is the stiffness matrix of the Laplacian operator for the pressure and
$M_p^{\circ} $ is the diagonal mass matrix given by
\begin{align}
    \label{mass-1}
    M_p^{\circ} = \text{diag}(|K_1|, ..., |K_N|),
\end{align}
where $K_j$ ($j = 1, ..., N$) denote the elements of $\mathcal{T}_h$ and $|K_j|$ denotes the volume of $K_j$.
From \eqref{A0-1} and \eqref{B-1}, it is not difficult to see that $A_0$ can be expressed as
\begin{equation}
    A_0 = (B^{\circ})^T (M_p^{\circ})^{-1} B^{\circ} .
\label{A0-2}
\end{equation}
% following from that
% \begin{align}
% \V{u}_h^T A_0 \V{v}_h & = \sum_{K\in\mathcal{T}_h} (\nabla_w\cdot\V{v}_h,\nabla_w\cdot\V{u}_h)_K \notag
% \\
% & =  ( B^{\circ} \V{v}_h )^T (M_p^{\circ})^{-1} B^{\circ} \V{u}_h \notag
% \\
% & 
% = \V{u}_h^T (B^{\circ})^T (M_p^{\circ})^{-1} B^{\circ} \V{v}_h,
% \quad \forall \V{u}_h, \V{v}_h \in \V{v}_h^0 .
% \label{A0-3}
% \end{align}
Combining this and Lemma~\ref{lem:B0-1}, we know that $A_0$ is singular. 
Moreover, we can rewrite \eqref{2by2Scheme_matrix} as
\begin{equation}
    \begin{bmatrix}
        \epsilon A_1  +  A_0 & -\frac{\alpha \epsilon}{\mu } B^T \\[0.05in]
        -\frac{\alpha \epsilon}{\mu} B & -\frac{\epsilon}{\mu }D
    \end{bmatrix}
    \begin{bmatrix}
        \V{u}_h \\
        \mathbf{p}_h
    \end{bmatrix}
    = \frac{\epsilon}{\mu}
    \begin{bmatrix}
        \mathbf{b}_1 \\
        \mathbf{b}_2
    \end{bmatrix},
    \qquad \mathcal{A} = \begin{bmatrix}
        \epsilon A_1  +  A_0 & -\frac{\alpha \epsilon}{\mu } B^T \\[0.05in]
        -\frac{\alpha \epsilon}{\mu} B & -\frac{\epsilon}{\mu }D
    \end{bmatrix}, 
    \label{2by2Scheme_matrix2}
\end{equation}
where $\epsilon = \frac{\mu}{ \lambda + \mu }$.
The above system is referred to as the two-field formulation of
the linear poroelasticity problem.
Since $A_0$ is only positive semi-definite,
the (1,1) block of the above system becomes nearly singular as
$\lambda \to \infty$ (i.e., $\epsilon \to 0$).
This near singularity poses challenges in developing effective preconditioners
and causes slow iterative convergence.
To overcome such difficulty, in the next section system~\eqref{2by2Scheme_matrix2} will be reformulated as a three-field formulation by introducing a numerical pressure, and an inherent regularization strategy will be added to the resulting system. 
In Section~\ref{sec:poro3}, the iterative solution of the regularized system using MINRES and GMRES will then be studied.

%%%%%%%%%%%%%%%%%%%%%%%%%%%%%%%%%%%%%%%%%%%%%%%%%%%%%%%%%%%%%%%
\section{The inherent regularization strategy}
\label{sec::reg}
In this section we convert \eqref{2by2Scheme_matrix2} to a three-field formulation  by introducing a numerical pressure variable and consider the inherent regularization strategy. We show that the regularized system is
equivalent to the original one.
Moreover, the regularization effectively eliminates the near singularity of the system.

To deal with the near singularity of the two-field system~\eqref{2by2Scheme_matrix2},
we introduce the numerical pressure as
\[
\V{z}_h = - (M_p^{\circ})^{-1} B^{\circ}\V{u}_h .
\]
Using this and \eqref{A0-2}, 
we can rewrite \eqref{2by2Scheme_matrix2} with the newly introduced unknown $\V{z}_h$
into a three-field formulation as
\begin{align*}
    \begin{cases}
        &\epsilon A_1 \V{u}_h  - \frac{\alpha \epsilon}{\mu} B^T \V{p}_h - (B^{\circ})^T \V{z}_h = \frac{\epsilon}{\mu} \V{b}_1,
        \\
        &
        -\frac{\alpha \epsilon}{\mu} B \V{u}_h - \frac{\epsilon}{\mu} D  \V{p}_h = \frac{\epsilon}{\mu} \V{b}_2,
        \\
        &
        -B^{\circ}\V{u}_h - M_p^{\circ} \V{z}_h = 0.
    \end{cases}
    \label{3by3Scheme0}
\end{align*}
Dividing the first two equations by $\epsilon$, we can cast the above system in the matrix-vector form as

\begin{equation}
    \begin{bmatrix}
      A_1  &  B^T & -(B^{\circ})^T\\[0.1in]
     B & \displaystyle -\frac{\mu}{ \alpha^2 }D & 0\\[0.1in]
     -B^{\circ} & 0& \displaystyle -\epsilon M_{p}^{\circ}
    \end{bmatrix}
    \begin{bmatrix}
        \mathbf{u}_h \\[0.1in]
        -\frac{\alpha}{\mu}\mathbf{p}_h \\[0.1in]
        \frac{1}{\epsilon} \mathbf{z}_h
    \end{bmatrix}
    =
     \begin{bmatrix}
       \frac{1}{\mu} \mathbf{b}_1 \\[0.1in]
       -\frac{1}{\alpha} \mathbf{b}_2 \\[0.1in]
       0
    \end{bmatrix} .
    \label{3by3Scheme}
\end{equation}
% A key feature of the WG method is that the element degrees of freedom
% can be eliminated locally from the linear system.
% To clearly present the regularization strategy,
We can eliminate the edge component $\mathbf{p}_h^{\partial}$ of $\mathbf{p}_h$
from the above system (and using (\ref{B-2}) and (\ref{D-2})), obtaining
\begin{align}
    \begin{bmatrix}
      A_1  &  (B^{\circ})^T & -(B^{\circ})^T\\[0.1in]
     B^{\circ} & \displaystyle -\frac{\mu}{ \alpha^2 }\tilde{D} & 0\\[0.1in]
     -B^{\circ} & 0& \displaystyle -\epsilon M_{p}^{\circ}
    \end{bmatrix}
    \begin{bmatrix}
        \mathbf{u}_h \\[0.1in]
        -\frac{\alpha}{\mu}\mathbf{p}_h^{\circ} \\[0.1in]
        \frac{1}{\epsilon} \mathbf{z}_h
    \end{bmatrix}
    =
     \begin{bmatrix}
       \frac{1}{\mu} \mathbf{b}_1 \\[0.1in]
       -\frac{1}{\alpha} \mathbf{b}_2^{\circ} \\[0.1in]
       0
    \end{bmatrix} ,
    \label{3field_eqn1}
\end{align}
where 
\begin{equation}
    \label{D-3}
    \tilde{D} = c_0 M_p^{\circ} + \kappa \Delta t \left (A_p^{\circ\circ}-A_p^{\circ\partial} (A_p^{\partial\partial})^{-1}A_p^{\partial\circ} \right ) .
\end{equation}
The above system becomes nearly singular for small values of $\epsilon$,
which continues to pose difficulties for iterative solvers to the poroelasticity problem.

We now consider the regularization for the above system. Define
\begin{align}
\label{e-1}
\V{1}= \frac{1}{\sqrt{N}}\begin{bmatrix} 1 \\ \vdots \\ 1 \end{bmatrix},
\end{align}
where $N$ is the number of elements in $\mathcal{T}_h$.
Recall that Lemma~\ref{lem:B0-1} implies $\V{1}^T B^{\circ} = \V{0}$.
Then, multiplying the third equation of \eqref{3field_eqn1} from the left by $\V{1}^T$,
we get
\[
(M_p^{\circ} \V{1})^T \V{z}_h = 0.
\]
Define
\begin{align}
    \label{w-1}
    \V{w} = \frac{M_p^{\circ} \V{1}}{\| M_p^{\circ} \V{1} \|}.
\end{align}
Then, for any positive parameter $\rho$ there holds
\begin{align}
\label{inh-1}
- \rho \V{w} \V{w}^T \V{z}_h = \V{0} .
\end{align}
Adding \eqref{inh-1} to the third block equation of \eqref{3field_eqn1}, we obtain
\begin{align}
    \begin{bmatrix}
      A_1  &  (B^{\circ})^T & -(B^{\circ})^T\\[0.1in]
     B^{\circ} & \displaystyle -\frac{\mu}{ \alpha^2 }\tilde{D} & 0\\[0.1in]
     -B^{\circ} & 0& \displaystyle -\epsilon M_{p}^{\circ} -\rho \V{w}\V{w}^T
    \end{bmatrix}
    \begin{bmatrix}
        \mathbf{u}_h \\[0.1in]
        -\frac{\alpha}{\mu}\mathbf{p}_h^{\circ} \\[0.1in]
        \frac{1}{\epsilon} \mathbf{z}_h
    \end{bmatrix}
    =
     \begin{bmatrix}
       \frac{1}{\mu} \mathbf{b}_1 \\[0.1in]
       -\frac{1}{\alpha} \mathbf{b}_2^{\circ} \\[0.1in]
       0
    \end{bmatrix} .
    \label{3field_eqn2}
\end{align}
The choice of $\rho$ will be discussed later this section.
% Here, we assume that $\rho$ is chosen
% as $\rho = \mathcal{O}(h^d)$; see (\ref{rho-1}) for a specific choice.
% This choice is needed to avoid small eigenvalues of the preconditioned Schur complement (cf. Lemma~\ref{lem:b-bound}).

By adding the third row to the second row and the third column to the second column, we get
\begin{align}
     \begin{bmatrix}
      A_1  & 0 & -(B^{\circ})^T\\[0.1in]
     0 & \displaystyle -\frac{\mu}{ \alpha^2 }\tilde{D} - \epsilon M_p^{\circ} - \rho \V{w}\V{w}^T & - \epsilon M_p^{\circ} - \rho \V{w}\V{w}^T\\[0.1in]
     -B^{\circ} & - \epsilon M_p^{\circ} - \rho \V{w}\V{w}^T& \displaystyle -\epsilon M_{p}^{\circ} -\rho \V{w}\V{w}^T
    \end{bmatrix}
     \begin{bmatrix}
        \mathbf{u}_h \\[0.1in]
        -\frac{\alpha}{\mu}\mathbf{p}_h^{\circ} \\[0.1in]
        \frac{1}{\epsilon} \mathbf{z}_h + \frac{\alpha}{\mu}\mathbf{p}_h^{\circ}
    \end{bmatrix}
    =
     \begin{bmatrix}
       \frac{1}{\mu} \mathbf{b}_1 \\[0.1in]
       -\frac{1}{\alpha} \mathbf{b}_2^{\circ} \\[0.1in]
       0
    \end{bmatrix} ,
    \label{3field_eqn3}
\end{align}
where 
\begin{align}
    \tilde{\mathcal{A}_3} = 
\left [
\begin{array}{c|cc}
     A_1  &  0 & -(B^{\circ})^T\\
     \noalign{\vskip 0.1in} \hline \noalign{\vskip 0.1in}
     \displaystyle  0 &-\frac{\mu}{ \alpha^2 }\tilde{D}-\epsilon M_p^{\circ} - \rho \V{w}\V{w}^T& -\epsilon M_p^{\circ}- \rho \V{w}\V{w}^T\\[0.1in]
     \displaystyle  -B^{\circ} & -\epsilon M_p^{\circ}- \rho \V{w}\V{w}^T & -\epsilon M_{p}^{\circ} - \rho \V{w}\V{w}^T\\    
\end{array} 
\right ] .
    \label{3by3Scheme-2}
\end{align}
This is the three-field formulation we will analyze for its iterative solution.
The Schur complement for this system is given by
\begin{align}
     \tilde{S}  =
         \begin{bmatrix}
           \displaystyle  \frac{\mu}{ \alpha^2 }\tilde{D}+\epsilon M_p^{\circ} + \rho \V{w}\V{w}^T & \epsilon M_p^{\circ} + \rho \V{w}\V{w}^T\\[0.1in]
     \displaystyle  \epsilon M_p^{\circ} + \rho \V{w}\V{w}^T & \epsilon M_{p}^{\circ} + \rho \V{w}\V{w}^T + B^{\circ} A_1^{-1} (B^{\circ})^T
       \end{bmatrix} .
\label{S3-1}
\end{align} 
We take the approximation to $\tilde{S}$ as
\begin{align}
    \hat{\tilde{S}} = 
    \begin{bmatrix}
        \displaystyle
        \frac{\mu}{\alpha^2} \tilde{D}+\epsilon M_p^{\circ} + \rho \V{w}\V{w}^T & 0 \\[0.1in]
        \displaystyle
        0 & M_p^{\circ} 
    \end{bmatrix}.
    \label{approxSchur-3field}
\end{align}
For block preconditioning of the system (\ref{3field_eqn3}), we need to carry out
the action of inversion of $\hat{\tilde{S}}$, which will be discussed
in Section~\ref{sec:poro3-implementation}.

%%%%%%%%%%%%%%%%%%%%%%%%%%%%%%%%%%%%%%%%%%%%%%%%%%%%%%%%%%%%%%%
%%%%%%%%%%%%%%%%%%%%%%%%%%%%%%%%%%%%%%%%%%%%%%%%%%%%%%%%%

The following two lemmas show that
$\tilde{S}$ is non-singular and the eigenvalues of $\hat{\tilde{S}}^{-1} \tilde{S}$
stay bounded below and above by positive constants.
As a consequence, the regularized system (\ref{3field_eqn3}) is non-singular.

\begin{lem}
    \label{lem:b-bound}
    Suppose that $\rho$ is taken as $\rho = C \lambda_{\min} (M_p^{\circ}),$ where $C$ is a constant independent of mesh size and parameters. Then, 
    the eigenvalues of $(M_p^\circ)^{-1} \Big ( \epsilon M_p^{\circ} + \rho \V{w}\V{w}^T + B^{\circ} A_1^{-1} (B^{\circ})^T \Big )$ lie in the interval 
    \begin{align}
            \Big[ C_1 + \epsilon + \mathcal{O}(N \rho^2), \; C_2 + \epsilon \Big],
            \label{lem:b-bound-1}
    \end{align}
    where
 \begin{align}
    \label{C4C5}
        C_1 = \min \{C \gamma^2, \beta^2 \} ,
         \quad
         C_2  = d + C,
         \quad \gamma = \V{w}^T \V{1}.
    \end{align}
\end{lem}

\begin{proof}
Following the proof of \cite[Lemma~5.1]{Huang-Wang-2025-Stokes-reg},
we obtain the bounds for the eigenvalues $\lambda_1 \le \lambda_2 \le \cdots \le \lambda_N$ of
$(M_p^\circ)^{-1} \Big ( \epsilon M_p^{\circ} + \rho \V{w}\V{w}^T + B^{\circ} A_1^{-1} (B^{\circ})^T \Big )$ as  
\begin{align}
& \lambda_1 = \epsilon + \frac{\gamma^2 N}{|\Omega|} \rho + \mathcal{O}(N \rho^2),
% \geq \epsilon + \frac{\gamma^2}{\lambda_{\max}(M_p^{\circ})}\rho + \mathcal{O}(N \rho^2) ,
\label{lem:b-bound-2}
\\
&   \epsilon + \beta^2 \le \lambda_i \le \epsilon +  d + \frac{\rho}{\lambda_{\min}(M_p^{\circ})},
    \quad i = 2, ..., N.
    \label{lem:b-bound-3}
\end{align}
On quasi-uniform meshes, with the choice $\rho = C \lambda_{\min} (M_p^{\circ})$, we have 
\[
\frac{\gamma^2 N}{|\Omega|} \rho 
\sim C\gamma^2.
\]
Thus, $\lambda_1 = \epsilon + C\gamma^2 + \mathcal{O}(N \rho^2)$.
The lower bound 
$\epsilon + \min\{C\gamma^2, \beta^2\} + \mathcal{O}(N \rho^2)$ follows from \eqref{lem:b-bound-2} and \eqref{lem:b-bound-3}.
The upper bound follows from \eqref{lem:b-bound-3}.

\end{proof}

\begin{lem}
\label{lem:eigen_bound_3field}
The eigenvalues of $\hat{\tilde{S}}^{-1} \tilde{S}$ lie in
\begin{align}
   \Bigg[ \frac{C_1}{1+C_1} \cdot  \frac{\frac{\mu}{\alpha^2} \lambda_{\min} (\tilde{D})}{\rho + \epsilon \lambda_{\max}(M_p^{\circ})
   + \frac{\mu}{\alpha^2} \lambda_{\min} (\tilde{D})} + \mathcal{O}(\epsilon) + \mathcal{O}(N \rho^2), \quad 1+ C_2 +\epsilon\Bigg],
    \label{lem:eigen_bound_3field-1}
\end{align}
where $C_1$ and $C_2$ are defined in (\ref{C4C5}).
\end{lem}

\begin{proof}
Consider the eigenvalue problem associated with $\hat{\tilde{S}}^{-1} \tilde{S}$,
\begin{align*}
&    \displaystyle
       \Big( \frac{\mu}{ \alpha^2 }\tilde{D}+\epsilon M_p^{\circ} + \rho \V{w}\V{w}^T \Big) \V{p}_h + \Big(\epsilon M_p^{\circ} + \rho \V{w}\V{w}^T\Big) \V{q}_h = \lambda \Big(\frac{\mu}{\alpha^2} \tilde{D}+\epsilon M_p^{\circ} + \rho \V{w}\V{w}^T\Big)\V{p}_h ,
       \\[0.1in]
&        \displaystyle
       \Big(\epsilon M_p^{\circ} + \rho \V{w}\V{w}^T\Big) \V{p}_h + \Big(\epsilon M_{p}^{\circ} + \rho \V{w}\V{w}^T + B^{\circ} A_1^{-1} (B^{\circ})^T\Big)\V{q}_h = \lambda M_p^{\circ} \V{q}_h.
\end{align*}
Since $\lambda = 1$ is contained in the interval (\ref{lem:eigen_bound_3field-1}), we 
assume that $\lambda \neq 1$ in the following analysis. Solving the first equation for $\V{p}_h$, we get
\begin{align*}
    \V{p}_h = \frac{ 1}{\lambda-1} \Big( \frac{\mu}{\alpha^2} \tilde{D} + \epsilon M_p^{\circ} + \rho \V{w}\V{w}^T\Big)^{-1}
    \Big(\epsilon M_p^{\circ} + \rho \V{w} \V{w}^T \Big)\V{q}_h.
\end{align*}
Substituting this into the second equation, we obtain
\begin{align*}
\displaystyle
   & \frac{\V{q}_h^T \Big( \epsilon  M_p^{\circ} + \rho \V{w}\V{w}^T \Big) \Big( \frac{\mu}{\alpha^2} \tilde{D} + \epsilon M_p^{\circ} + \rho \V{w}\V{w}^T\Big)^{-1}  \Big( \epsilon  M_p^{\circ} + \rho \V{w}\V{w}^T \Big)\V{q}_h}{\V{q}_h^T M_p^{\circ}  \V{q}_h }
   \notag \\
   & \quad \quad \qquad
    + (\lambda - 1)
    \frac{\V{q}_h^T \Big( \epsilon  M_p^{\circ} + \rho \V{w}\V{w}^T +  B^{\circ} A_1^{-1} (B^{\circ})^T\Big)\V{q}_h}{\V{q}_h^T M_p^{\circ} \V{q}_h}
    =
    \lambda (\lambda-1).
\end{align*}
Denoting
\begin{align*}
    & a =  \frac{\V{q}_h^T \Big( \epsilon  M_p^{\circ} + \rho \V{w}\V{w}^T \Big) \Big( \frac{\mu}{\alpha^2} \tilde{D} + \epsilon M_p^{\circ} + \rho \V{w}\V{w}^T\Big)^{-1} \Big( \epsilon  M_p^{\circ} + \rho \V{w}\V{w}^T \Big)\V{q}_h}{\V{q}_h^T M_p^{\circ}  \V{q}_h }, \notag
    \\
    & b =  \frac{\V{q}_h^T \Big( \epsilon  M_p^{\circ} + \rho \V{w}\V{w}^T +  B^{\circ} A_1^{-1} (B^{\circ})^T\Big)\V{q}_h}{\V{q}_h^T M_p^{\circ} \V{q}_h} ,
\end{align*}
we can rewrite the above equation as
\begin{align*}
    \lambda^2 - (1+b) \lambda + (b-a) = 0,
\end{align*}
which has the roots
\begin{align*}
    \lambda_{+} = \frac{(1+b) + \sqrt{(1+b)^2 - 4(b-a)}}{2},\quad
    \lambda_{-} = \frac{(1+b) - \sqrt{(1+b)^2 - 4(b-a)}}{2}.
\end{align*}
Since both $\hat{\tilde{S}}$ and $\tilde{S}$ are symmetric,
$\lambda_{-}$ and $\lambda_{+}$ are real. Moreover, it is not difficult to show that $a \le b$.
Then, we have
\begin{align}
\frac{b-a}{1+b} \le \lambda_{-} \le \lambda_{+} \le 1 + b.
\label{eigen_bound-eq-1}
\end{align}
Notice that the lower and upper bounds for $b$ can be obtained from Lemma~\ref{lem:b-bound}, i.e.,
\begin{align}
C_1 + \epsilon + \mathcal{O}(N \rho^2) \le b \le C_2 + \epsilon .
\label{eigen_bound-eq-2}
\end{align}

To obtain a lower bound for $\lambda_{-}$, we want to get an upper bound for $a/b$,
which can be obtained by establishing an upper bound of the largest eigenvalue of the matrix
\[
\Big( \epsilon  M_p^{\circ} + \rho \V{w}\V{w}^T \Big)^{\frac{1}{2}} \Big( \frac{\mu}{\alpha^2} \tilde{D} + \epsilon M_p^{\circ} + \rho \V{w}\V{w}^T\Big)^{-1} \Big( \epsilon  M_p^{\circ} + \rho \V{w}\V{w}^T \Big)^{\frac{1}{2}},
\]
or the reciprocal of the minimum eigenvalue of the matrix
\[
\Big( \epsilon  M_p^{\circ} + \rho \V{w}\V{w}^T \Big)^{-1} \Big( \frac{\mu}{\alpha^2} \tilde{D} + \epsilon M_p^{\circ} + \rho \V{w}\V{w}^T\Big).
\]
Thus, we have
\begin{align}
\frac{a}{b} \le \frac{1}{1 + \frac{\mu}{\alpha^2} \frac{\lambda_{\min}(\tilde{D})}{\rho + \epsilon \lambda_{\max}(M_p^{\circ})}}
    =  \frac{ \rho + \epsilon \lambda_{\max}(M_p^{\circ})}{ \rho +
    \epsilon \lambda_{\max}(M_p^{\circ}) + \frac{\mu}{\alpha^2} \lambda_{\min} (\tilde{D})}.
    \label{eigen_bound-eq-3}
\end{align}
Using (\ref{eigen_bound-eq-1}), (\ref{eigen_bound-eq-2}), and (\ref{eigen_bound-eq-3}), we obtain
\eqref{lem:eigen_bound_3field-1}.
\end{proof}

For quasi-uniform meshes, we have ${\lambda_{\max}(M_p^{\circ})}
= \mathcal{O}(h^d)$, ${\lambda_{\min}(M_p^{\circ})}
= \mathcal{O}(h^d)$, and ${\lambda_{\min}(\tilde{D})}
= c_0\mathcal{O}(h^d) + \kappa\Delta t \mathcal{O}(h^d)$ (e.g., see \cite{KamenskiHuangXu_2014}).
Moreover, $C_1$ and $C_2$ are constants when $\rho$ is chosen as 
$\rho = C \lambda_{\min}(M_p^{\circ})$.
In this case, Lemma~\ref{lem:eigen_bound_3field} implies that
the eigenvalues of $\hat{\tilde{S}}^{-1} \tilde{S}$ are bounded above and below essentially by positive constants.

\section{Convergence of MINRES and GMRES}
\label{sec:poro3}
In this section we study the iterative solution of the three-field formulation (\ref{3field_eqn3}) of
the linear poroelasticity problem 
using MINRES and GMRES with inexact block Schur complement preconditioning.

\subsection{Convergence analysis of MINRES and GMRES with block preconditioners}
\label{sec:poro3-analysis}

We can use MINRES with a block diagonal preconditioner $\tilde{\mathcal{P}}_{d}$
and GMRES with a block triangular preconditioner $\tilde{\mathcal{P}}_{t}$
to solve (\ref{3field_eqn3}), where
\begin{align}
    & \tilde{\mathcal{P}}_{d} = 
    \begin{bmatrix}
        A_1 & 0 \\ 0 & \hat{\tilde{S}} 
    \end{bmatrix}
    = 
    \begin{bmatrix}
        A_1 & 0 & 0\\
        0 &         \frac{\mu}{\alpha^2} \tilde{D}+\epsilon M_p^{\circ} + \rho \V{w}\V{w}^T & 0 \\
   0 & 0 & M_{p}^{\circ}
    \end{bmatrix},
    \label{3field-diag-0}
\\
& \tilde{\mathcal{P}}_{t} =
    \begin{bmatrix}
        A_1 & 0 \\ \begin{bmatrix} 0 \\ - B^{\circ} \end{bmatrix}
        & - \hat{\tilde{S}} 
    \end{bmatrix}
    = 
    \begin{bmatrix}
        A_1 & 0 & 0\\
        0 &  - \frac{\mu}{\alpha^2} \tilde{D}-\epsilon M_p^{\circ} - \rho \V{w}\V{w}^T  & 0 \\
    -B^{\circ} & 0 & -M_{p}^{\circ}
    \end{bmatrix} .
    \label{3field-tri-0}
\end{align}

We now show that the convergence of MINRES and GMRES with block preconditioners is independent of the parameters.

\begin{lem}
\label{lem:eigen_bound_diag_poro}
The eigenvalues of $ \mathcal{\tilde{P}}_{d}^{-1} \mathcal{\tilde{A}}_{3} $ lie in
$[-a_1,-b_1] \cup [c_1,d_1]$, where
\begin{align*}
  & a_1 = \frac{1}{2} \big (\sqrt{4(1+C_2)+1}+1\big ) + \mathcal{O}(\epsilon) ,
  \\
  & b_1 =  \frac{1}{2} \big (\sqrt{ \frac{4 C_1}{1+C_1} \cdot \frac{\frac{\mu}{\alpha^2} \lambda_{\min} (\tilde{D})}{\rho + \epsilon \lambda_{\max}(M_p^{\circ})
   + \frac{\mu}{\alpha^2} \lambda_{\min} (\tilde{D})} + 1} - 1\big ) + \mathcal{O}(\epsilon) ,
   \\
   & c_1 = \frac{1}{2} \big (\sqrt{ \frac{4 C_1}{1+C_1} \cdot \frac{\frac{\mu}{\alpha^2} \lambda_{\min} (\tilde{D})}{\rho + \epsilon \lambda_{\max}(M_p^{\circ})
   + \frac{\mu}{\alpha^2} \lambda_{\min} (\tilde{D})} + 1} - 1\big ) + \mathcal{O}(\epsilon),  
 \\
   & d_1 = \frac{1}{2} \big (\sqrt{4(1+C_2)+1}+1\big ) + \mathcal{O}(\epsilon),
%    \label{eigen_bound_diag}
\end{align*}
and $C_1$ and $C_2$ are defined in (\ref{C4C5}).
\end{lem}

\begin{proof}
The eigenvalue problem associated with the preconditioned system $\mathcal{\tilde{P}}^{-1}_{d}\mathcal{\tilde{A}}_{3}$  reads as
   \begin{align*}
    \begin{bmatrix}
        A_1 & -C^T \\[0.05in]
        -C & -\tilde{S} + 
        \begin{pmatrix}
            0 & 0 \\ 0 & B^{\circ} A_1^{-1} (B^{\circ})^T
        \end{pmatrix}
        \end{bmatrix}
    \begin{bmatrix}
       \V{u}_h \\
        \V{w}_h
    \end{bmatrix} = \lambda
   \begin{bmatrix}
        A_1 & 0 \\
        0 & \hat{\tilde{S}}
    \end{bmatrix} 
        \begin{bmatrix}
      \V{u}_h \\
        \V{w}_h
    \end{bmatrix} ,
\end{align*}
where $ C = \begin{bmatrix}
    0 \\ B^{\circ}
\end{bmatrix}.$ 
This leads to the system
\begin{align*}
\begin{cases}
      A_1 \V{u}_h - C^T \V{w}_h = \lambda A_1 \V{u}_h \\
      -C\V{u}_h - \tilde{S} \V{w}_h  + 
        \begin{pmatrix}
            0 & 0 \\ 0 & B^{\circ} A_1^{-1} (B^{\circ})^T
        \end{pmatrix}\V{w}_h = \lambda \hat{\tilde{S}} \V{w}_h.
\end{cases}
\end{align*}
For $\lambda \neq 1$, the first equation yields $\V{u}_h = \frac{-1}{\lambda - 1} A_1^{-1} C^T \V{w}_h$.
Plugging this into the second equation, we have
\begin{align*}
    \lambda^2 \hat{\tilde{S}} \V{w}_h - \lambda \Big(\hat{\tilde{S}} - \tilde{S} + 
        \begin{pmatrix}
            0 & 0 \\ 0 & B^{\circ} A_1^{-1} (B^{\circ})^T
        \end{pmatrix} \Big) \V{w}_h - \tilde{S} \V{w}_h = \V{0},
\end{align*}
which yields
\begin{align*}
        \lambda^2 - \lambda \frac{\V{w}_h^T\Big(\hat{\tilde{S}} - \tilde{S} + \begin{pmatrix}
            0 & 0 \\ 0 & B^{\circ} A_1^{-1} (B^{\circ})^T
        \end{pmatrix}\Big) \V{w}_h}{\V{w}_h^T \hat{\tilde{S}} \V{w}_h} - \frac{\V{w}_h^T \tilde{S} \V{w}_h}{\V{w}_h^T\hat{\tilde{S}} \V{w}_h} = 0.
\end{align*}
Denoting 
\begin{align*}
    & a = \frac{\V{w}_h^T\Big(\hat{\tilde{S}} - \tilde{S} +   \begin{pmatrix}
            0 & 0 \\ 0 & B^{\circ} A_1^{-1} (B^{\circ})^T
        \end{pmatrix} \Big) \V{w}_h}{\V{w}_h^T \hat{\tilde{S}} \V{w}_h}, \\
        & b = \frac{\V{w}_h^T \tilde{S} \V{w}_h}{\V{w}_h^T\hat{\tilde{S}} \V{w}_h},
\end{align*}
we have 
\begin{align*}
    \lambda^2 - a\lambda - b = 0,
\end{align*}
which has the roots
\begin{align*}
    \lambda_{+} = \frac{a + \sqrt{a^2 + 4b}}{2},\quad
    \lambda_{-} = \frac{a - \sqrt{a^2 + 4b}}{2}.
\end{align*}
From Lemma~\ref{lem:eigen_bound_3field}, 
\[
\frac{C_1}{1+C_1} \cdot  \frac{\frac{\mu}{\alpha^2} \lambda_{\min} (\tilde{D})}{\rho + \epsilon \lambda_{\max}(M_p^{\circ})
   + \frac{\mu}{\alpha^2} \lambda_{\min} (\tilde{D})} + \mathcal{O}(\epsilon) + \mathcal{O}(N \rho^2) \leq b \leq  1+ C_2 +\epsilon.
\]
% It is not difficult to show that
It can be shown that
\[
-1 \leq a \leq 1.
\]
The conclusion follows from these bounds on $a$ and $b$ together with the expressions for $\lambda_{\pm}$.
\end{proof}

\begin{thm}
\label{pro:MINRES_conv-DBC}
The residual of MINRES applied to a preconditioned system associated
with the coefficient matrix $\tilde{\mathcal{P}}_{d}^{-1} \tilde{\mathcal{A}}_{3} $ is bounded by   
\begin{align}
    \frac{\| \V{r}_{2k}\| }{\| \V{r}_0 \|} 
    % \; {\stackrel{<}{\sim}} \; 
    \leq 2 \left( \frac{\sqrt{\frac{a_1d_1}{b_1 c_1}} - 1}{\sqrt{\frac{a_1d_1}{b_1 c_1}} + 1}\right)^k,
     \label{pro:MINRES_conv-DBC-1}
\end{align}
where $a_1$, $b_1$, $c_1$, and $d_1$ are given in Lemma~\ref{lem:eigen_bound_diag_poro}.
\end{thm}

\begin{proof}
The result follows from Lemma~\ref{lem:eigen_bound_diag_poro} and \cite[Theorem 6.13]{Elman-2014}. 
\end{proof}

\begin{thm}
\label{pro:gmres_conv_DBC}
The residual of GMRES applied to the preconditioned system $ \tilde{\mathcal{P}}_{t}^{-1} \tilde{\mathcal{A}}_{3} $ is bounded by
\begin{align}
\frac{\| \V{r}_k\|}{\| \V{r}_0\|} 
% \; {\stackrel{<}{\sim}} \; 
\le
&
2\Bigg(2 + C_{2} + \epsilon  + \sqrt{ \frac{d\, \lambda_{\max} (M_p^{\circ})}{\lambda_{\min} (A_1)}} \Bigg) 
\notag \\
& \times \left( \frac{ \sqrt{1+C_2} - \sqrt{ \frac{C_1}{1+C_1} \cdot  \frac{\frac{\mu}{\alpha^2} \lambda_{\min} (\tilde{D})}{\rho + \epsilon \lambda_{\max}(M_p^{\circ}) 
   + \frac{\mu}{\alpha^2} \lambda_{\min} (\tilde{D})}}     }{  \sqrt{ 1+C_2 }   + \sqrt{\frac{C_1}{1+C_1} \cdot  \frac{\frac{\mu}{\alpha^2} \lambda_{\min} (\tilde{D})}{\rho + \epsilon \lambda_{\max}(M_p^{\circ}) 
   + \frac{\mu}{\alpha^2} \lambda_{\min} (\tilde{D})}}  } + \mathcal{O}(\epsilon) + \mathcal{O}(N \rho^2)  \right)^{k-1}  .
\label{pro:gmres_conv_DBC-0}
\end{align}
\end{thm}

\begin{proof}
From \cite[Lemma~A.1]{HuangWang_CiCP_2025}, we know that the residual of GMRES for the preconditioned system $\tilde{\mathcal{P}}_{t}^{-1} \tilde{\mathcal{A}}_{3} $ is bounded as 
    \begin{align}
        \frac{\| \V{r}_k \|}{\|\V{r}_0\|} \le
        (1+\|A_1^{-1}(B^{\circ})^T\| + \| \hat{\tilde{S}}^{-1} \tilde{S} \|) \min\limits_{\substack{p \in \mathbb{P}_{k-1}\\ p(0) = 1}} \| p(\hat{\tilde{S}}^{-1} \tilde{S}) \| .
        \label{eigen_bound_tri_poro-DBC-1}
    \end{align}
Lemma~\ref{lem:eigen_bound_3field} implies that
\[
\| \hat{\tilde{S}}^{-1} \tilde{S} \| \leq 1 + C_2 +\epsilon .
\]
Moreover, using the fact that both $A_1$ and $M_p^{\circ}$ are SPD, and 
\begin{align*}
   \sum_{K \in \mathcal{T}_h}(\nabla_w \cdot \V{u}_h,\nabla_w \cdot \V{u}_h)_K \le
d \sum_{K \in \mathcal{T}_h}(\nabla_w \V{u}_h,\nabla_w \V{u}_h )_K ,
\end{align*}
we have
\begin{align*}
  \| A_1^{-1} (B^{\circ})^T \|^2
    & \leq \sup_{\V{p}_h \neq 0} \frac{\V{p}_h^T B^{\circ} A_1^{-1} A_1^{-1} (B^{\circ})^T \V{p}_h}{\V{p}_h^T  \V{p}_h} \nonumber
    \\
   & = \sup_{\V{p}_h \neq 0} \frac{\V{p}_h^T (M_p^{\circ})^{\frac{1}{2}} (M_p^{\circ})^{-\frac{1}{2}}B^{\circ} A_1^{-1}A_1^{-1} (B^{\circ})^T (M_p^{\circ})^{-\frac{1}{2}} (M_p^{\circ})^{\frac{1}{2}} \V{p}_h}{\V{p}_h^T \V{p}_h}
   \notag \\
   & \le \lambda_{\max} (A_1^{-1}) \lambda_{\max} (M_p^{\circ})
 \sup_{\V{u}_h \neq 0} \frac{\V{u}_h^T (B^{\circ})^T (M_p^{\circ})^{-1} B^{\circ} \V{u}_h}{\V{u}_h^T A_1 \V{u}_h}
 \notag \\
   & = \lambda_{\max} (A_1^{-1}) \lambda_{\max} (M_p^{\circ})
 \sup_{\V{u}_h \neq 0} \frac{\V{u}_h^T A_0 \V{u}_h}{\V{u}_h^T A_1 \V{u}_h}
 \notag \\
 & \le \frac{d \, \lambda_{\max} (M_p^{\circ})}{\lambda_{\min} (A_1)} .
\end{align*}
% For the minmax problem in \eqref{eigen_bound_tri_poro-DBC-1}, by shifted Chebyshev polynomials
% (e.g., see \cite[Pages 50-52]{Greenbaum-1997}) and Lemma~\ref{lem:eigen_bound_3field}
% (where we denote the lower and upper bounds by $a$ and $b$, respectively), we have
For the minmax problem in~\eqref{eigen_bound_tri_poro-DBC-1}, by using shifted Chebyshev polynomials
(see, e.g.,~\cite[pp.~50--52]{Greenbaum-1997}) and Lemma~\ref{lem:eigen_bound_3field}
(where the lower and upper bounds are denoted by $a$ and $b$, respectively), we obtain
\begin{align*}
   \min\limits_{\substack{p \in \mathbb{P}_{k-1}\\ p(0) = 1}} \| p( \hat{\tilde{S}}^{-1} \tilde{S} ) \|
& = \min\limits_{\substack{p \in \mathbb{P}_{k-1}\\ p(0) = 1}} \max_{i=1,..., N} |p(\lambda_i( \hat{\tilde{S}}^{-1} \tilde{S} ))|
\le \min\limits_{\substack{p \in \mathbb{P}_{k-1}\\ p(0) = 1}} \max_{\gamma \in \left[a,b\right]} |p(\gamma)| 
\\
& \leq 2 \left(\frac{\sqrt{b} - \sqrt{a}}{\sqrt{b} + \sqrt{a}} \right)^{k-1} ,
\end{align*}
which leads to the result in \eqref{pro:gmres_conv_DBC-0}.
\end{proof}

It is worth noting that, when $\rho$ is taken according to Lemma~\ref{lem:b-bound},
the right-hand sides of (\ref{pro:MINRES_conv-DBC-1}) and 
(\ref{pro:gmres_conv_DBC-0}) are essentially constants for small $\epsilon$
and $h$ on quasi-uniform meshes. The above two theorems imply that the convergence
of the preconditioned MINRES and GMRES methods -- in terms of both the convergence factor
and the asymptotic convergence constant -- is essentially independent of $h$, $\epsilon$
(or $\lambda$), and $\kappa \Delta t$.

\subsection{Implementation}
\label{sec:poro3-implementation}

Notice that the analysis in the previous subsection is for the reduced system \eqref{3field_eqn3}.
It is more convenient to implement preconditioned MINRES and GMRES on the original system
including $\V{p}_h^{\partial}$, i.e., 
\begin{equation}
    \mathcal{A}_3
    \begin{bmatrix}
        \mathbf{u}_h \\[0.05in]
        -\frac{\alpha}{\mu}\mathbf{p}_h \\[0.05in]
        \frac{1}{\epsilon} \mathbf{z}_h + \frac{\alpha}{\mu}\mathbf{p}_h^{\circ}
    \end{bmatrix}
    =
     \begin{bmatrix}
       \frac{1}{\mu} \mathbf{b}_1 \\[0.05in]
       -\frac{1}{\alpha} \mathbf{b}_2 \\[0.05in]
       0
    \end{bmatrix} ,
    \label{3field_eqn4}
\end{equation}
where
\begin{align*}
        \qquad 
    \mathcal{A}_3 = \begin{bmatrix}
      A_1  &  0 & -(B^{\circ})^T\\[0.1in]
     \displaystyle  0 &-\frac{\mu}{\alpha^2}D - \epsilon \begin{pmatrix}
         M_p^{\circ} & 0 \\
         0 & 0
     \end{pmatrix}- \begin{pmatrix}
         \rho \V{w}\V{w}^T & 0 \\
         0 & 0
     \end{pmatrix}
     & -\epsilon \begin{pmatrix} M_p^{\circ}\\ 0 \end{pmatrix} - \begin{pmatrix}\rho \V{w}\V{w}^T\\ 0 \end{pmatrix} \\[0.1in]
     \displaystyle  -B^{\circ} & -\epsilon \begin{pmatrix} M_p^{\circ} & 0 \end{pmatrix} - \begin{pmatrix} \rho\V{w}\V{w}^T & 0 \end{pmatrix} & -\epsilon M_{p}^{\circ} - \rho \V{w}\V{w}^T
    \end{bmatrix} .
\end{align*}
We can find 
the Schur complement $S$ and its approximation $\hat{S}$ (corresponding to (\ref{approxSchur-3field})) as
\begin{align}
     &S  = 
         \begin{bmatrix}
           \displaystyle  \frac{\mu}{\alpha^2}D  +\epsilon \begin{pmatrix}
         M_p^{\circ} & 0 \\
         0 & 0
     \end{pmatrix}+ \begin{pmatrix}
         \rho \V{w}\V{w}^T & 0 \\
         0 & 0
     \end{pmatrix} & \epsilon \begin{pmatrix} M_p^{\circ} \\ 0 \end{pmatrix} + \begin{pmatrix} \rho\V{w}\V{w}^T \\ 0 \end{pmatrix} \\[0.1in]
     \displaystyle  \epsilon \begin{pmatrix} M_p^{\circ} & 0 \end{pmatrix} + \begin{pmatrix} \rho\V{w}\V{w}^T & 0 \end{pmatrix}  & \epsilon M_{p}^{\circ} + \rho \V{w}\V{w}^T + B^{\circ} A_1^{-1} (B^{\circ})^T
       \end{bmatrix} ,
       \label{S3}
       \\[0.2in]
       & \hat{S} = 
       \begin{bmatrix}
           \displaystyle  \frac{\mu}{\alpha^2}D  +\epsilon \begin{pmatrix}
         M_p^{\circ} & 0 \\
         0 & 0
     \end{pmatrix}+ \begin{pmatrix}
         \rho \V{w}\V{w}^T & 0 \\
         0 & 0
     \end{pmatrix} & 0 \\[0.1in]
     \displaystyle  0&  M_{p}^{\circ} 
       \end{bmatrix} .
\label{hatS3}
\end{align} 
The corresponding block diagonal and triangular preconditioners become
\begin{align}
&    \mathcal{P}_{d} = 
    \begin{bmatrix}
        A_1 & 0 & 0\\
        0 &         \displaystyle \frac{\mu}{\alpha^2}D +\epsilon \begin{pmatrix}
         M_p^{\circ} & 0 \\ 0 & 0
    \end{pmatrix} +  \begin{pmatrix}
         \rho \V{w}\V{w}^T & 0 \\ 0 & 0
    \end{pmatrix} & 0 \\
   0 & 0 & M_{p}^{\circ}
    \end{bmatrix},
    \label{3field-diag}
\\
&     \mathcal{P}_{t} = 
    \begin{bmatrix}
        A_1 & 0 & 0\\
        0 &  \displaystyle - \frac{\mu}{\alpha^2}D - \epsilon \begin{pmatrix}
         M_p^{\circ} & 0 \\ 0 & 0
    \end{pmatrix} -  \begin{pmatrix}
         \rho \V{w}\V{w}^T & 0 \\
         0 & 0 \end{pmatrix}   & 0 \\
    -B^{\circ} & 0 & -M_{p}^{\circ}
    \end{bmatrix} .
    \label{3field-tri}
\end{align}
As is the case for $\tilde{\mathcal{P}}_{d}^{-1}\tilde{\mathcal{A}}_3$ and $\tilde{\mathcal{P}}_{t}^{-1}\tilde{\mathcal{A}}_3$, it can be shown that the convergence of MINRES applied to $\mathcal{P}_{d}^{-1}\mathcal{A}_3$ and GMRES applied to $\mathcal{P}_{t}^{-1}\mathcal{A}_3$ is essentially independent of $h$, $\lambda$, and $\kappa \Delta t$.
The proof is omitted for brevity.

For the implementation of $\mathcal{P}_{d}$ and $\mathcal{P}_{t}$, we need
to carry out the action of inversion of the diagonal blocks.
For $A_1$, this can be done using CG with an incomplete Cholesky decomposition of $A_1$.
For the middle block, this can also be done using CG but with an incomplete Cholesky
decomposition of
\[
\frac{\mu}{\alpha^2}D +\epsilon \begin{pmatrix}
         M_p^{\circ} & 0 \\ 0 & 0
    \end{pmatrix} .
\]
Note that the term $\rho \V{w} \V{w}^T$ is omitted when forming the incomplete Cholesky decomposition.
Our limited experience shows that this strategy works
well, with PCG converging typically just in a few iterations.
The inversion of $M_p^{\circ}$ is trivial since it is diagonal.

%%%%%%%%%%%%%%%%%%%%%%%%%%%%%%%%%%%%%%%%%%%%%%
%%%%%%%%%%%%%%%%%%%%%%%%%%%%%%%%%%%%%%%%%%%%%%%%%%%%%%%%%%%%%%%
\section{Numerical experiments}
\label{SEC:numerical}

In this section, we present 2D and 3D numerical results
for poroelasticity problems to demonstrate the performances of MINRES and GMRES
with block preconditioning. 

We take the regularization constant as $\rho = 0.1\, \lambda_{\min}(M_p^{\circ})$, satisfying Lemma~\ref{lem:b-bound}, and
use MATLAB's functions {\em minres} and {\em gmres} with a relative tolerance $tol = 10^{-8}$ for both 2D
and 3D examples, a maximum of 1000 iterations, and the zero vector as the initial guess.
Moreover, $restart = 30$ is used with {\em gmres}.
The implementation of block preconditioners $\mathcal{P}_{d}$ and $\mathcal{P}_{t}$ 
is performed as discussed in Section~\ref{sec:poro3-implementation}.
The inversion of the leading and middle blocks is done using CG preconditioned with {\em ichol} 
with threshold dropping and a drop tolerance of $10^{-3}$.

%%%%%%%%%%%%%%%%%%%%%%%%%%%%%%%%%%%%%%%%%%%%%%%%%%%%%\subsection{Linear poroelasticity in three-field formulation}

\subsection{Constant $\kappa$}
% \subsection{Smooth solution}
\label{ex-smooth}

The 2D and 3D linear poroelasticity examples are modifications of an example
of \cite{SISCLeePierMarRog2019} and 
the right-hand side functions are given in 2D by
\begin{align*}
\mathbf{f} & = 
 - t \begin{bmatrix}
        -8 \pi^2 \mu \cos(2 \pi x) \sin(2 \pi y) - \frac{2 \pi^2 \mu}{\lambda + \mu} \sin(\pi x) \sin(\pi y) \\
      + 4 \pi^2 \mu \sin(2\pi y) 
      + \pi^2 \cos(\pi x+\pi y)
      + \alpha \pi \cos(\pi x)\sin(\pi y) 
      \\[0.08in]
        8 \pi^2 \mu \sin(2 \pi x) \cos(2 \pi y) - \frac{2 \pi^2 \mu}{\lambda + \mu} \sin(\pi x) \sin(\pi y) \\
      - 4 \pi^2 \mu \sin(2\pi x) 
      + \pi^2 \cos(\pi x+\pi y)
      + \alpha \pi \sin(\pi x)\cos(\pi y)
\end{bmatrix},
\\ %[0.08in]
s & = -c_0 \sin(\pi x)\sin(\pi y) 
+ \frac{\pi \alpha}{\lambda + \mu} \sin(\pi x+\pi y)
	  -t \kappa \Big(2 \pi^2 \sin(\pi x) \sin(\pi y)\Big) ,
\end{align*}
and in 3D by
\begin{align*}
\mathbf{f} & = t
 \begin{bmatrix}
 4\mu \cos(2\pi x) \sin(2\pi y) \sin(2\pi z) \pi^2 + \frac{(4\mu + \lambda)}{(\mu + \lambda)} \sin(\pi x) \sin(\pi y) \sin(\pi z) \pi^2 \\
- \cos(\pi x) \cos(\pi y) \sin(\pi z) \pi^2 - \cos(\pi x) \sin(\pi y) \cos(\pi z) \pi^2 \\
+ 8\pi^2 \mu \left(-1 + \cos(2\pi x)\right) \sin(2\pi y) \sin(2\pi z) + \alpha \pi \cos(\pi x) \sin(\pi y) \sin(\pi z)
\\[0.08in]
- \pi^2 \cos(\pi x) \cos(\pi y) \sin(\pi z) + \frac{(4\mu + \lambda)}{(\mu + \lambda)} \sin(\pi x) \sin(\pi y) \sin(\pi z) \pi^2 \\
- \sin(\pi x) \cos(\pi y) \cos(\pi z) \pi^2 + 16\pi^2 \mu \sin(2\pi x) \left(1 - \cos(2\pi y)\right) \sin(2\pi z) \\
- 8\pi^2 \mu \sin(2\pi x) \cos(2\pi y) \sin(2\pi z) + \alpha \pi \sin(\pi x) \cos(\pi y) \sin(\pi z)
\\[0.08in]
 \frac{(4\mu + \lambda)}{(\mu + \lambda)} \sin(\pi x) \sin(\pi y) \sin(\pi z) \pi^2 + 4\pi^2 \mu \sin(2\pi x) \sin(2\pi y) \cos(2\pi z) \\
- \cos(\pi x) \sin(\pi y) \cos(\pi z) \pi^2 - \sin(\pi x) \cos(\pi y) \cos(\pi z) \pi^2 \\
+ 8\pi^2 \mu \left(-1 + \cos(2\pi z)\right) \sin(2\pi x) \sin(2\pi y) + \alpha \pi \sin(\pi x) \sin(\pi y) \cos(\pi z)
\end{bmatrix},
\\
s &= \frac{\alpha \pi}{\mu + \lambda} \Big( \cos(\pi x) \sin(\pi y) \sin(\pi z) + \sin(\pi x) \cos(\pi y) \sin(\pi z)
\\
& \qquad \qquad + \sin(\pi x) \sin(\pi y) \cos(\pi z) \Big)
 + (3 \pi^2 t \kappa + c_0)\sin(\pi x) \sin(\pi y) \sin(\pi z) .
\end{align*}
The parameters are taken as $c_0 = 1$, $\kappa \in \{10^{-6},\,10^{-3},\,1\}$,
$\lambda \in\{ 1, 10^2, 10^4, 10^6\}$, $\mu = 1$, and $\Delta t \in \{ 10^{-6}, 10^{-3}, 10^{-2}\}$.

Tables~\ref{Poro-3field-2D-v2} and~\ref{Poro-3field-3D-v2} report the iteration counts
of MINRES and GMRES for the preconditioned systems in two and three dimensions, respectively. 
Zero initial guesses are used for the iterative solvers. 
The iteration counts remain essentially constant with respect to mesh refinement
and variations in $\kappa$ and $\Delta t$. 
While the number of iterations changes slightly as $\lambda$ increases from $1$ to $10^{2}$,
it remains nearly constant for $\lambda = 10^{2}$, $10^{4}$, and $10^{6}$.

Table~\ref{Poro-3field-2D-v2-noreg} reports the iteration counts of MINRES and GMRES
for the unregularized system in the two-dimensional example.  For this case, the convergence
of MINRES and GMRES with the corresponding inexact block preconditioners was analyzed
in~\cite{HuangWang_2025}, where it was shown that the convergence factor is essentially independent
of the mesh size and the locking parameter. However, a few additional iterations are required
to address the small eigenvalue of the preconditioned system as $\epsilon \to 0$.

The results in Table~\ref{Poro-3field-2D-v2-noreg} indicate that, for $\lambda = 1$, the unregularized formulation requires nearly the same number of iterations as the regularized formulation reported in Table~\ref{Poro-3field-2D-v2}.
For larger values of $\lambda$, however, the unregularized formulation requires slightly more iterations,
with the iteration count increasing as $\lambda$ increases.
More specifically, with regularization, GMRES converges in approximately 25 iterations across all tested cases,
regardless of variations in $\lambda$, $\Delta t$, $\kappa$, and $h$. 
In contrast, without regularization, the GMRES iteration count increases with $\lambda$,
nearly doubling as $\lambda$ increases from $1$ to $10^{6}$. 
For MINRES, although the improvement due to regularization is less pronounced
than for GMRES, the iteration counts remain more stable with regularization than without.
Table~\ref{Poro-3field-2D-reg-time} shows that the solving time for the regularized 
system scales consistently with the mesh size and follows $\mathcal{O}(N^{\frac{3}{2}})$ 
across all mesh sizes. This is consistent with the theory that given the nearly constant 
outer MINRES/GMRES iteration counts and the $\mathcal{O}(h^{-1}) = \mathcal{O}(N^{\frac{1}{2}})$ 
growth of the inner PCG iterations \cite{Greenbaum-1997}, each costing $\mathcal{O}(N)$ 
per iteration, the total solving time scales as 
$ \mathcal{O}(N^{\frac{3}{2}})$.
We also observe that the solving time is consistent with the 
number of MINRES/GMRES iterations in Table~\ref{Poro-3field-2D-v2}. 
Without 
regularization, the solving time shown in Table~\ref{Poro-3field-2D-noreg-time} 
is also consistent with the number of iterations in 
Table~\ref{Poro-3field-2D-v2-noreg}. Moreover, GMRES takes slightly less time to converge with regularization than without regularization.

As indicated in Lemma~\ref{lem:b-bound}, the regularization parameter $\rho$ is chosen as
$\rho = C \lambda_{\min}(M_p^{\circ})$, where $C$ is a constant independent of the mesh size $h$ and the physical parameters. In addition to the choice $\rho = 0.1 \lambda_{\min}(M_p^{\circ})$ used in both the 2D and 3D examples, we also consider $\rho = 0.01 \lambda_{\min}(M_p^{\circ})$ and $\rho = \lambda_{\min}(M_p^{\circ})$, as reported in Tables~\ref{Poro-3field-2D-rho001} and \ref{Poro-3field-2D-rho1}, respectively. Across these choices of $\rho$, both MINRES and GMRES exhibit comparable and consistent iteration counts.

To investigate how the solution accuracy of the inner PCG affects the performance of the outer solver MINRES, we report the iteration counts in Table~\ref{table-block-inverse} for various relative tolerances of the inner PCG. The results show that the iteration count remains stable when the PCG relative tolerance is less than or equal to $10^{-6}$, which is larger than the $10^{-8}$ tolerance used for the outer MINRES solver. Nevertheless, for robustness, we set the relative tolerance of the inner PCG to be $0.01$ times that used for the outer solver.

% matlab code: ex2d_poroelasticity_WG3field

\begin{table}[htb!]
    \centering
    \caption{The 2D example for linear poroelasticity with regularization ($\rho = 0.1 \lambda_{\min} (M_p^{\circ})$): The number of MINRES and GMRES iterations required to reach convergence for preconditioned systems $\mathcal{P}_{d}^{-1} \mathcal{A}_3$ and $\mathcal{P}_{t}^{-1} \mathcal{A}_3$, respectively, with various values of $\lambda $, $\kappa$, and $\Delta t $.}
        \begin{tabular}{|c|c|c|c|c|c|c||c|c|c|c|c|}
        \hline
       & & & \multicolumn{8}{c|}{$N$} \\ \cline{4-11}
       $\kappa$ & $\Delta t$ & $\lambda$ & 918 & 3680 & 14728 & 58608 & 918 & 3680 & 14728 & 58608 \\ \hline
       & & & \multicolumn{4}{c||}{MINRES} & \multicolumn{4}{c|}{GMRES}\\ \hline
       \multirow{12}{*}{$1$}   &  \multirow{4}{*}{$10^{-6}$}  & $1$ & 44 & 44 & 46 & 48 & 24& 23 & 23 & 23\\ 
       & & $10^{2}$  & 46 &  51& 54 & 58 & 24 & 25 & 26 & 27 \\ 
       & & $10^{4}$  & 48 & 52 & 56 & 58 & 25 &26&27  &28 \\ 
        & & $10^{6}$ &48  &52  &56  & 58 & 25&26  & 27 & 28\\ 
       \cline{2-11}
        &  \multirow{4}{*}{$ 10^{-3}$}  & $1$ & 41 & 43 & 43 &45  & 22 & 22 & 21 & 21\\ 
       & & $10^{2}$  &  46& 51 & 54 & 58 &24  & 25 & 26 & 27\\ 
       & & $10^{4}$  & 48 & 52 & 56 & 58 & 25&26  &27&28 \\ 
        & & $10^{6}$ & 48 & 52 & 56 & 58 &25 & 26 & 27 & 28 \\ 
       \cline{2-11}
        &  \multirow{4}{*}{$ 10^{-2}$}  & $1$ & 38 &  38& 38 & 40 &19 & 19 & 19 & 19 \\ 
       & & $10^{2}$  & 46 & 51 &54  &  58&24 & 25 & 26 & 27\\ 
       & & $10^{4}$  & 48 & 52 &  55& 58 & 25 & 26 & 27 & 28 \\ 
        & & $10^{6}$ & 48 & 52 &  55& 58  &25 & 26 &27  &28 \\ 
        \hline
       \multirow{12}{*}{$10^{-3}$}  &  \multirow{4}{*}{$ 10^{-6}$}  & $1$ & 44 & 44 & 46 & 46  &23 & 23 &23  &23 \\ 
       & & $10^{2}$  & 46 & 51 & 54 &  58& 24 & 25  & 26 & 27\\ 
       & & $10^{4}$ & 48 & 52 &56  & 58 & 25& 26 & 27 & 28\\ 
        & & $10^{6}$ & 48 & 52 &56  & 58 &25 & 26 & 27 & 28\\ 
       \cline{2-11}
        &  \multirow{4}{*}{$ 10^{-3}$}  & $1$ &44  & 46 & 46 & 48 & 24 & 23 & 23& 23 \\ 
       & & $10^{2}$  & 46 & 51& 54 &  58 & 24 &  25 & 26 & 27 \\ 
       & & $10^{4}$  & 48 & 52 &56& 58 & 25 & 26 & 27 &28 \\ 
        & & $10^{6}$ &48  & 52 & 56 &58  &25  & 26 & 27 & 28\\ 
       \cline{2-11}
        &  \multirow{4}{*}{$ 10^{-2}$}  & $1$ & 44 &46  &  46&  48&24 & 23 &23  &23 \\ 
       & & $10^{2}$  & 46 & 51 &  54& 58 & 24 & 25  &  26& 27\\ 
       & & $10^{4}$ & 48 & 52 &  56& 58  & 25& 26 & 27 & 28\\ 
        & & $10^{6}$ & 48 & 52 &  56& 58 &25 & 26 & 27 & 28\\ 
        \hline
       \multirow{12}{*}{$10^{-6}$}  &  \multirow{4}{*}{$10^{-6}$}  & $1$ & 44 &  44&  46& 46 &24 & 23 &23  &23 \\ 
       & & $10^{2}$  & 46 &  51& 54 & 58 & 24 & 25 &  26& 27\\ 
       & & $10^{4}$  & 48 &52  & 56 &  58& 25& 26 & 27 &28 \\ 
        & & $10^{6}$ &  48& 52 & 56 & 58 & 25&  26& 27 & 28\\ 
       \cline{2-11}
        &  \multirow{4}{*}{$ 10^{-3}$}  & $1$ &44  & 44 & 46 & 46 & 24& 23 & 23 &23 \\ 
       & & $10^{2}$  & 46 & 51 &  54&  58& 24 & 25 & 26 & 27 \\ 
       & & $10^{4}$  & 48 & 52 & 56 & 58 & 25& 26 &  27& 28\\ 
        & & $10^{6}$ & 48 & 52 & 56 & 58 & 25& 26 &  27& 28\\  
       \cline{2-11}
        &  \multirow{4}{*}{$10^{-2}$} & $1$ & 44 &44  & 46 & 46 & 24  & 23 & 23&23\\ 
       & & $10^{2}$  & 46 & 51 &  54&58  &24  &25 & 26 & 27\\ 
       & & $10^{4}$  & 48 & 52 & 56 &  58& 25& 26 & 27 & 28\\ 
        & & $10^{6}$ & 48 & 52 & 56 & 58 &25 & 26 &  27&28 \\ 
        \hline
    \end{tabular}
    \label{Poro-3field-2D-v2}
\end{table}

% matlab code: ex3d_poroelasticity_WG3field

\begin{table}[htb!]
    \centering
    \caption{The 3D example for linear poroelasticity with regularization ($\rho = 0.1 \lambda_{\min} (M_p^{\circ})$): The number of MINRES and GMRES iterations required to reach convergence for preconditioned systems $\mathcal{P}_{d}^{-1} \mathcal{A}_3$ and $\mathcal{P}_{t}^{-1} \mathcal{A}_3$, respectively, with various values of $\lambda$, $\kappa$, and $\Delta t $.}
        \begin{tabular}{|c|c|c|c|c|c|c||c|c|c|c|c|}
        \hline
       & & & \multicolumn{8}{c|}{$N$} \\ \cline{4-11}
       $\kappa$ & $\Delta t$ & $\lambda$ & 7938 & 65171 & 526031 & 1777571 & 7938 & 65171 & 526031 & 1777571 \\ \hline
       & & & \multicolumn{4}{c||}{MINRES} & \multicolumn{4}{c|}{GMRES}\\ \hline
       \multirow{12}{*}{$ 1$}   &  \multirow{4}{*}{$10^{-6}$}  & $1$ &53  & 55 & 58 & 59 & 29 & 29 & 29 &29 \\ 
       & & $10^{2}$  &72  & 76 & 82 & 86& 39& 41 & 42 & 43\\ 
       & & $10^{4}$  & 74 & 78 & 84 & 88 & 40& 42 &  43&45 \\ 
        & & $10^{6}$ & 74 & 78 &84  & 88 & 40& 42 &  43&45 \\ 
       \cline{2-11}
        &  \multirow{4}{*}{$ 10^{-3}$}  & $1$ & 50 &52  &  54& 54 &27 & 27 & 27 &27\\ 
       & & $10^{2}$  &72  & 76 & 82 & 86&39 & 41 & 42 & 43\\ 
       & & $10^{4}$  & 74 & 78 & 84 & 88 & 40& 42 &43  & 45\\ 
        & & $10^{6}$ & 74 & 78 & 84 & 88 & 40& 42 &43  & 44  \\ 
       \cline{2-11}
        &  \multirow{4}{*}{$ 10^{-2}$}  & $1$ &44  &  45& 47 &49  &24 & 24 & 24 &24 \\ 
       & & $10^{2}$  & 72 & 74 & 82 & 86& 39& 41 &42  &43 \\ 
       & & $10^{4}$  & 74 & 78 & 84 & 88 & 40& 42 &43  & 45\\ 
        & & $10^{6}$ & 74 & 78 & 84 &  88& 40& 42 &43  & 44 \\ 
        \hline
       \multirow{12}{*}{$ 10^{-3}$}  &  \multirow{4}{*}{$ 10^{-6}$}  & $1$ & 53 & 55 & 58 & 59 &29 & 29 & 29 &29 \\ 
       & & $10^{2}$  & 72 & 76 & 82 &86 & 39& 41 &42  & 43\\ 
       & & $10^{4}$  & 74 & 78 &  84&88   & 40& 42 & 43 & 45\\ 
        & & $10^{6}$ & 74 & 78 &  84&88   & 40& 42 & 43 & 45\\  
       \cline{2-11}
        &  \multirow{4}{*}{$ 10^{-3}$}  & $1$ & 50 & 52 & 54 & 54 & 27& 27 & 27 & 27\\ 
       & & $10^{2}$ & 72 & 76 & 82  & 86& 39& 41 &42  & 43\\ 
       & & $10^{4}$   & 74 & 78 &  84&88 & 40 & 42 & 43 & 45\\ 
        & & $10^{6}$  & 74 & 78 &  84&88  & 40 & 42 & 43 & 45\\ 
       \cline{2-11}
        &  \multirow{4}{*}{$ 10^{-2}$}  & $1$ & 44 &  45& 47 & 49 & 24& 24 & 24 & 24\\ 
       & & $10^{2}$  & 72 & 74 & 82 & 86& 39& 41 &42  & 43 \\ 
       & & $10^{4}$   & 74 & 78 &  84&88 & 40& 42 & 43 & 45\\ 
        & & $10^{6}$  & 74 & 78 &  84&88  & 40& 42 & 43 & 45 \\ 
        \hline
       \multirow{12}{*}{$ 10^{-6}$}  &  \multirow{4}{*}{$ 10^{-6}$}  & $1$ & 53 & 55 & 58 & 59 & 29& 29 & 29 & 29\\ 
       & & $10^{2}$ & 72 & 76 & 82 & 86& 39& 41 &42  & 43\\ 
       & & $10^{4}$   & 74 & 78 &  84&88 & 40& 42 & 43 & 45\\ 
        & & $10^{6}$  & 74 & 78 &  84&88  & 40& 42 & 43 & 45\\ 
       \cline{2-11}
        &  \multirow{4}{*}{$ 10^{-3}$}  & $1$ &50  & 52 & 54 &  54&27 & 27 &27  & 27\\ 
       & & $10^{2}$  & 72 & 76 & 82 &86 & 39& 41 &42  & 43 \\ 
       & & $10^{4}$   & 74 & 78 &  84&88 & 40& 42 & 43 & 45\\ 
        & & $10^{6}$  & 74 & 78 &  84&88  & 40& 42 & 43 & 45 \\ 
       \cline{2-11}
        &  \multirow{4}{*}{$ 10^{-2}$} & $1$ & 44 &  45& 47 & 49 & 27 & 27 &27  & 27\\ 
       & & $10^{2}$  & 72 & 74 & 82 & 86& 39& 41 &42  & 43\\ 
       & & $10^{4}$   & 74 & 78 &  84&88 & 40& 42 & 43 & 45\\ 
        & & $10^{6}$  & 74 & 78 &  84&88  & 40& 42 & 43 & 45 \\ 
        \hline
    \end{tabular}
    \label{Poro-3field-3D-v2}
\end{table}

\begin{table}[htb!]
    \centering
    \caption{ The 2D example for linear poroelasticity without regularization: The number of MINRES and GMRES iterations required to reach convergence for preconditioned systems $\mathcal{P}_{d}^{-1} \mathcal{A}_3$ and $\mathcal{P}_{t}^{-1} \mathcal{A}_3$, respectively, with various values of $\lambda $, $\kappa$, and $\Delta t $.}
        \begin{tabular}{|c|c|c|c|c|c|c||c|c|c|c|c|}
        \hline
       & & & \multicolumn{8}{c|}{$N$} \\ \cline{4-11}
       $\kappa$& $\Delta t$ & $\lambda$ & 918 & 3680 & 14728 & 58608 & 918 & 3680 & 14728 & 58608 \\ \hline
       & & & \multicolumn{4}{c||}{MINRES} & \multicolumn{4}{c|}{GMRES}\\ \hline
       \multirow{12}{*}{$1$}   &  \multirow{4}{*}{$ 10^{-6}$}  & $1$ & 44 & 46 & 46 & 48 & 24 & 23 & 23 & 23\\ 
       & & $10^{2}$  & 54 & 55 & 59 & 61 & 27 & 28 &  28&29 \\ 
       & & $10^{4}$  &64  &68  & 68 & 74 &32 &  35& 33 &34 \\ 
        & & $10^{6}$ & 40 & 44 & 46 & 48 &49 & 48 & 50 &49 \\ 
       \cline{2-11}
        &  \multirow{4}{*}{$ 10^{-3}$}  & $1$ & 41 & 43 & 43 &45  & 22& 22 & 21 &21 \\ 
       & & $10^{2}$ & 54 & 55 &  58&  61&27 & 28 & 28 & 29\\ 
       & & $10^{4}$  & 64 & 68 & 68 & 74 &32 & 35 & 33 & 34\\ 
        & & $10^{6}$ & 40 & 44 & 46 & 48 & 49&48  & 50 & 49\\ 
       \cline{2-11}
        &  \multirow{4}{*}{$ 10^{-2}$}  & $1$ & 38 & 38 & 38 & 40& 19& 19 & 19 & 19\\ 
       & & $10^{2}$  & 54 & 55 &  58&  61&27 & 28 & 28 & 29\\ 
       & & $10^{4}$  & 64 &68  & 68 & 72 &32 & 35 & 33 & 34\\ 
        & & $10^{6}$ & 40 &44& 46 & 48 &49 & 48 & 50 & 49\\ 
        \hline
       \multirow{12}{*}{$10^{-3}$}  &  \multirow{4}{*}{$ 10^{-6}$}  & $1$ & 44 & 46 & 46 & 48 & 24& 23 & 23 & 23\\ 
       & & $10^{2}$  & 54 & 55 & 59 &  61&27 & 28 & 28 & 29 \\ 
       & & $10^{4}$  & 64 & 68 & 68 &  74&32 & 35 & 33 & 34\\ 
        & & $10^{6}$ & 40 &  44& 46 & 48 & 49 & 48 & 50 & 49\\  
       \cline{2-11}
        &  \multirow{4}{*}{$ 10^{-3}$}  & $1$ & 44 & 46 & 46 & 48 & 24& 23 & 23 & 23\\ 
       & & $10^{2}$  &  54& 55 & 59 & 61 &27 & 28 & 28 & 29 \\ 
       & & $10^{4}$  & 64 & 68 & 68 & 74 & 32& 35 & 33 &34 \\ 
        & & $10^{6}$ & 40 & 44 & 46 & 48 & 49&48  & 50 & 49\\ 
       \cline{2-11}
        &  \multirow{4}{*}{$ 10^{-2}$}  & $1$ & 44 & 46 & 46 & 48 & 23& 23 &23  &23 \\ 
       & & $10^{2}$  &  54& 55 & 58 & 61 &27 & 28 & 28 & 29 \\ 
       & & $10^{4}$  & 64 & 68 & 68 & 74 & 32& 35 & 33 &34 \\ 
        & & $10^{6}$ & 40 &44  &46& 48 & 49 & 48 & 50 &49 \\ 
        \hline
       \multirow{12}{*}{$10^{-6}$}  &  \multirow{4}{*}{$ 10^{-6}$}  & $1$ & 44 & 46 &  46& 48 & 24& 23 & 23 & 23\\ 
       & & $10^{2}$  &54  & 55 & 59 & 61&27 & 28 & 28 & 29 \\ 
       & & $10^{4}$  & 64 & 68 &68  & 74 &32 & 35 & 33 &34 \\ 
        & & $10^{6}$ &40  &  44& 46 & 48 &49 &48  &  50& 49\\ 
       \cline{2-11}
        &  \multirow{4}{*}{$ 10^{-3}$}  & $1$ & 44 & 46 & 46 & 48 & 24 & 23  & 23 & 23\\ 
       & & $10^{2}$  & 54 & 55 & 59 &  61&27 & 28 & 28 & 29 \\ 
       & & $10^{4}$  & 64 &  68& 68 & 74 & 32 &35  &  33& 34\\ 
        & & $10^{6}$ & 40 &44  &46  & 48 & 49 & 48 &  50& 49 \\  
       \cline{2-11}
        &  \multirow{4}{*}{$ 10^{-2}$} & $1$ &44  & 46 & 46 & 48 & 24& 23 & 23 & 23\\ 
       & & $10^{2}$  & 54 & 55 &  59&61  &27 & 28 & 28 & 29 \\ 
       & & $10^{4}$  & 64 & 68 & 68 & 74 &32 & 35 & 33 &34 \\ 
        & & $10^{6}$ & 40 & 44 & 46 & 48 &49 &48  &  50& 49\\
        \hline
    \end{tabular}
    \label{Poro-3field-2D-v2-noreg}
\end{table}

\begin{table}[htb!]
\tiny
    \centering
    \caption{The 2D example for linear poroelasticity with regularization {($\rho = 0.1 \lambda_{\min} (M_p^{\circ})$)}:
    The CPU time (in seconds) for MINRES and GMRES required to reach convergence for preconditioned systems $\mathcal{P}_{d}^{-1} \mathcal{A}_3$ and $\mathcal{P}_{t}^{-1} \mathcal{A}_3$, respectively, with various values of $\lambda $, $\kappa$, and $\Delta t $.}
        \begin{tabular}{|c|c|c|c|c|c|c||c|c|c|c|c|}
        \hline
       & & & \multicolumn{8}{c|}{$N$} \\ \cline{4-11}
       $\kappa$ & $\Delta t$ & $\lambda$ & 918 & 3680 & 14728 & 58608 & 918 & 3680 & 14728 & 58608 \\ \hline
       & & & \multicolumn{4}{c||}{MINRES} & \multicolumn{4}{c|}{GMRES}\\ \hline
       \multirow{12}{*}{$1$}   &  \multirow{4}{*}{$10^{-6}$}  & $1$ &   2.378e-1  &8.058e-1 & 4.674e+0 &  3.612e+1 & 1.375e-1 & 5.348e-1 &  2.645e+0 &1.901e+1  \\ 
       & & $10^{2}$  & 4.715e-1  & 8.271e-1  &  5.510e+0 &4.521e+1  &9.685e-2  &4.563e-1 &  2.787e+0  & 2.119e+1 \\ 
       & & $10^{4}$  & 1.509e-1  &  7.646e-1 & 5.291e+0 & 4.181e+1 & 3.066e-1 &  5.273e-1 &  2.704e+0 &  2.207e+1\\ 
        & & $10^{6}$ &  1.429e-1 &8.143e-1  & 5.357e+0  & 4.234e+1 & 9.650e-2 &4.143e-1  & 2.761e+0 &  2.081e+1\\  
       \cline{2-11}
        &  \multirow{4}{*}{$ 10^{-3}$}  & $1$ &1.381e-1  &  7.688e-1  &  4.524e+0 &  3.708e+1& 9.942e-2 &4.705e-1 &2.403e+0   & 1.813e+1 \\ 
       & & $10^{2}$  & 1.537e-1 & 8.968e-1  & 5.902e+0 & 5.101e+1  &1.034e-1 & 4.916e-1 &  3.064e+0& 2.357e+1\\
       & & $10^{4}$  & 1.851e-1  & 9.057e-1  &  6.139e+0  & 5.067e+1 &1.067e-1 & 5.341e-1& 3.056e+0 &2.490e+1 \\
        & & $10^{6}$ & 1.531e-1  & 8.820e-1   &  6.156e+0& 5.024e+1 &9.111e-2  &4.787e-1 &3.171e+0  & 2.422e+1\\
       \cline{2-11}
        &  \multirow{4}{*}{$ 10^{-2}$}  & $1$ & 1.233e-1 & 7.097e-1 &  4.741e+0  &4.064e+1   & 8.806e-2 &4.020e-1  & 2.543e+0 &  2.097e+1\\
       & & $10^{2}$  & 1.580e-1  &  9.347e-1  & 7.414e+0  &  6.295e+1 &  9.556e-2 &5.885e-1 &   3.611e+0& 3.103e+1\\
       & & $10^{4}$  &  1.711e-1& 1.058e+0  & 7.378e+0  &  6.377e+1 & 9.697e-2  &7.037e-1  & 3.686e+0  & 3.187e+1\\ 
        & & $10^{6}$ & 1.627e-1  &  9.262e-1 & 7.502e+0 &   6.584e+1& 1.095e-1&  5.856e-1 & 3.780e+0  & 3.249e+1 \\ 
        \hline
       \multirow{12}{*}{$10^{-3}$}  &  \multirow{4}{*}{$ 10^{-6}$}  & $1$ & 1.342e-1  & 7.398e-1 &   4.570e+0 & 3.503e+1 &8.729e-2  & 3.922e-1 &  2.427e+0 &  1.807e+1 \\ 
       & & $10^{2}$  &  1.393e-1  &7.197e-1   &5.220e+0  & 4.349e+1  & 8.830e-2  &  4.273e-1 &  2.695e+0 & 2.080e+1\\
       & & $10^{4}$ & 1.415e-1  & 7.285e-1  &  5.409e+0  &  4.271e+1 & 8.299e-2 & 4.786e-1  & 2.740e+0  & 2.076e+1  \\
        & & $10^{6}$ & 1.379e-1  & 9.402e-1  &5.420e+0   & 4.336e+1 & 9.400e-2 & 4.574e-1 & 2.763e+0  & 2.107e+1\\
       \cline{2-11}
        &  \multirow{4}{*}{$ 10^{-3}$}  & $1$ &1.433e-1  & 6.685e-1  & 4.674e+0  & 3.738e+1 & 9.016e-2  & 4.372e-1&  2.481e+0  &  1.805e+1\\
       & & $10^{2}$  &   1.354e-1 &  8.022e-1 &  5.181e+0  &  4.264e+1  & 1.187e-1&  4.456e-1&2.719e+0  & 2.100e+1\\
       & & $10^{4}$ & 1.306e-1 &  7.075e-1& 5.363e+0  & 4.206e+1 &9.001e-2  & 5.625e-1  & 2.696e+0&   2.096e+1\\
        & & $10^{6}$ &  1.149e-1  & 7.112e-1  & 5.441e+0 &  4.215e+1& 8.022e-2& 4.596e-1 &  2.663e+0 &2.120e+1  \\
       \cline{2-11}
        &  \multirow{4}{*}{$ 10^{-2}$}  & $1$ & 1.286e-1 &  7.931e-1  &  4.838e+0 & 3.728e+1 & 9.191e-2 & 4.623e-1&  2.519e+0& 1.845e+1\\
       & & $10^{2}$  & 1.304e-1  &6.921e-1 & 5.255e+0  &4.291e+1   & 8.516e-2& 4.190e-1 &  2.704e+0 & 2.043e+1 \\ 
       & & $10^{4}$ & 1.402e-1 & 7.811e-1  &  5.357e+0 & 4.265e+1 & 8.085e-2 &  5.909e-1& 2.711e+0 & 2.116e+1 \\
        & & $10^{6}$ & 1.252e-1 & 7.943e-1 & 5.586e+0  & 4.235e+1  &8.943e-2 &5.433e-1 &2.726e+0   &  2.088e+1\\ 
        \hline
       \multirow{12}{*}{$10^{-6}$}  &  \multirow{4}{*}{$10^{-6}$}  & 1 &  4.002e-1 &  6.734e-1 & 4.655e+0  & 3.405e+1 & 8.129e-2&4.170e-1  & 2.396e+0& 1.756e+1\\
       & & $10^{2}$  & 1.520e-1 &  7.799e-1&   5.272e+0&4.343e+1  & 8.229e-2 &4.366e-1  & 2.678e+0  & 2.014e+1\\
       & & $10^{4}$  &  1.740e-1 & 8.178e-1  &  5.598e+0 &4.322e+1   & 8.579e-2  & 4.603e-1& 2.804e+0 & 2.106e+1\\
        & & $10^{6}$ & 1.245e-1  &  7.575e-1 &5.443e+0  &  4.286e+1 &8.883e-2  & 4.673e-1& 2.702e+0 &  2.056e+1 \\
       \cline{2-11}
        &  \multirow{4}{*}{$ 10^{-3}$}  & $1$ & 1.351e-1  &  6.722e-1 &   4.664e+0 & 3.465e+1 & 8.391e-2  &  3.947e-1& 2.414e+0 & 1.789e+1 \\
       & & $10^{2}$  &  1.911e-1& 7.635e-1  &  5.360e+0  & 4.357e+1  & 7.887e-2& 4.364e-1  & 2.708e+0 & 2.088e+1\\
       & & $10^{4}$  & 1.551e-1  & 8.025e-1 & 5.459e+0  &  4.346e+1  &8.978e-2  & 4.689e-1&2.723e+0  & 2.139e+1 \\
        & & $10^{6}$&  1.320e-1 & 8.339e-1 & 5.533e+0  &  4.273e+1  &7.918e-2 &4.783e-1  & 2.748e+0 & 2.142e+1\\
       \cline{2-11}
        &  \multirow{4}{*}{$10^{-2}$} & $1$ &1.429e-1  &   7.339e-1&  4.797e+0 &  3.421e+1 & 9.067e-2 & 4.156e-1  & 2.470e+0 & 1.789e+1\\
       & & $10^{2}$  & 1.371e-1 & 8.622e-1 &  5.506e+0  & 4.223e+1 &  7.814e-2& 5.708e-1 &  2.750e+0 & 2.005e+1\\
       & & $10^{4}$  & 1.382e-1 &8.798e-1   &  5.747e+0  & 4.175e+1 & 8.280e-2  &5.008e-1  &  2.849e+0 &  2.124e+1\\
        & & $10^{6}$ & 1.635e-1 &  7.827e-1 &5.704e+0   &  4.179e+1  &1.023e-1  &  4.849e-1&  2.877e+0 &2.127e+1  \\
        \hline
    \end{tabular}
    \label{Poro-3field-2D-reg-time}
\end{table}

\begin{table}[htb!]
    \centering
    \tiny
    \caption{The 2D example for linear poroelasticity without regularization:
    The CPU time (in seconds) for MINRES and GMRES required to reach convergence for preconditioned systems $\mathcal{P}_{d}^{-1} \mathcal{A}_3$ and $\mathcal{P}_{t}^{-1} \mathcal{A}_3$, respectively, with various values of $\lambda $, $\kappa$, and $\Delta t $.}
        \begin{tabular}{|c|c|c|c|c|c|c||c|c|c|c|c|}
        \hline
       & & & \multicolumn{8}{c|}{$N$} \\ \cline{4-11}
       $\kappa$ & $\Delta t$ & $\lambda$ & 918 & 3680 & 14728 & 58608 & 918 & 3680 & 14728 & 58608 \\ \hline
       & & & \multicolumn{4}{c||}{MINRES} & \multicolumn{4}{c|}{GMRES}\\ \hline
       \multirow{12}{*}{$1$}   &  \multirow{4}{*}{$10^{-6}$}  & $1$ & 2.422e-1 &  7.570e-1 & 4.478e+0 & 3.272e+1 & 1.049e-1 &4.633e-1 & 2.492e+0 & 1.829e+1\\ 
       & & $10^{2}$  & 1.460e-1 & 7.830e-1 & 4.916e+0 & 3.839e+1 &1.052e-1 & 4.779e-1& 2.844e+0  & 2.163e+1 \\ 
       & & $10^{4}$  &1.734e-1  & 9.351e-1  & 5.597e+0 & 4.681e+1 & 1.367e-1 & 6.127e-1&3.279e+0   &  2.509e+1\\ 
        & & $10^{6}$ & 1.105e-1 &  6.165e-1  & 3.841e+0 & 2.943e+1 &1.872e-1 & 1.028e+0& 5.374e+0 & 3.217e+1\\  
       \cline{2-11}
        &  \multirow{4}{*}{$ 10^{-3}$}  & $1$ & 1.517e-1 & 7.280e-1 & 3.964e+0 & 3.255e+1  & 9.083e-2& 4.176e-1 & 2.246e+0  & 1.574e+1  \\ 
       & & $10^{2}$  & 1.943e-1  & 1.004e+0 & 5.596e+0 & 4.456e+1 & 9.785e-2 & 5.116e-1 & 2.824e+0  &2.263e+1 \\
       & & $10^{4}$  &  4.347e-1&  1.045e+0 & 6.453e+0  &5.361e+1   & 1.251e-1& 6.498e-1  & 3.517e+0 &2.695e+1  \\
        & & $10^{6}$ & 1.168e-1  & 8.575e-1 & 4.294e+0 & 3.627e+1  &2.219e-1  & 8.812e-1 & 5.010e+0 & 3.759e+1 \\
       \cline{2-11}
        &  \multirow{4}{*}{$ 10^{-2}$}  & $1$ & 1.249e-1  &  6.772e-1  & 4.155e+0 & 3.706e+1 & 7.216e-2 &  3.759e-1& 2.243e+0  & 1.793e+1\\
       & & $10^{2}$  & 1.708e-1 & 9.675e-1 & 6.628e+0  & 5.834e+1  & 1.071e-1 & 6.026e-1&  3.459e+0 & 2.877e+1 \\
       & & $10^{4}$  & 2.257e-1 & 1.364e+0  & 7.640e+0&6.692e+1  &1.378e-1 & 7.428e-1& 4.167e+0 &3.419e+1 \\ 
        & & $10^{6}$ & 1.385e-1  & 7.665e-1  &5.438e+0  & 4.901e+1  & 2.024e-1& 9.876e-1&  6.197e+0& 4.882e+1  \\ 
        \hline
       \multirow{12}{*}{$10^{-3}$}  &  \multirow{4}{*}{$ 10^{-6}$}  & $1$ & 1.190e-1 & 6.550e-1  & 3.953e+0 &  3.142e+1  & 8.077e-2& 3.939e-1 & 2.120e+0 &  1.554e+1\\ 
       & & $10^{2}$  &  1.428e-1& 7.606e-1 & 5.00e+0 & 3.934e+1 &9.345e-2 &6.185e-1  & 2.511e+0  & 1.960e+1\\
       & & $10^{4}$ & 1.852e-1  & 9.272e-1  & 5.659e+0 &4.684e+1   &1.204e-1  & 6.006e-1&  2.999e+0  &  2.254e+1\\
        & & $10^{6}$ & 1.029e-1 & 6.113e-1  &3.896e+0   &  3.179e+1& 1.713e-1& 7.727e-1& 4.788e+0  &3.202e+1  \\
       \cline{2-11}
        &  \multirow{4}{*}{$ 10^{-3}$}  & $1$ & 1.207e-1 & 6.784e-1 & 4.421e+0  & 3.173e+1  &1.006e-1 &3.886e-1  & 2.298e+0  & 1.647e+1\\
       & & $10^{2}$  & 1.334e-1  &7.583e-1  & 4.926e+0  & 3.819e+1 &7.891e-2  &4.375e-1 & 2.645e+0 & 1.902e+1 \\
       & & $10^{4}$ & 1.597e-1  & 9.025e-1 & 5.551e+0  &  4.607e+1 & 1.015e-1 & 5.367e-1 & 2.978e+0  &2.237e+1  \\
        & & $10^{6}$ & 9.770e-2 & 5.834e-1  &  3.762e+0  & 3.041e+1 & 1.506e-1 &  7.288e-1& 4.328e+0  & 3.381e+1 \\
       \cline{2-11}
        &  \multirow{4}{*}{$ 10^{-2}$}  & $1$ & 1.316e-1 & 6.963e-1 & 4.239e+0  &  3.207e+1  & 7.614e-2 & 5.605e-1& 2.241e+0 &  1.618e+1  \\
       & & $10^{2}$  &3.422e-1  &  7.474e-1 & 4.838e+0 & 3.870e+1  &8.757e-2 & 4.452e-1& 2.533e+0 & 1.939e+1 \\ 
       & & $10^{4}$ & 1.624e-1 & 9.106e-1  &  5.543e+0 & 4.648e+1 &9.617e-2 & 5.473e-1& 3.108e+0  &  2.295e+1 \\
        & & $10^{6}$ & 9.675e-2 & 5.820e-1  & 3.833e+0  & 3.057e+1   & 1.532e-1& 7.255e-1  &4.577e+0  & 3.267e+1  \\ 
        \hline
       \multirow{12}{*}{$10^{-6}$}  &  \multirow{4}{*}{$10^{-6}$}& $1$ & 1.100e-1  &6.299e-1  & 4.029e+0 & 3.030e+1  &6.547e-2  &3.454e-1  &  2.014e+0 & 1.525e+1\\
       & & $10^{2}$  & 1.331e-1 & 7.326e-1 & 4.842e+0 & 3.819e+1  &7.663e-2 & 4.309e-1 & 2.419e+0 & 1.972e+1\\
       & & $10^{4}$  & 1.569e-1  &  8.845e-1 &  5.512e+0  & 4.606e+1  & 9.439e-2& 5.610e-1 & 3.059e+0  & 2.337e+1\\
        & & $10^{6}$ & 1.028e-1 & 7.599e-1  &  3.782e+0& 2.925e+1 &1.480e-1  & 7.199e-1 & 4.242e+0  & 3.327e+1 \\
       \cline{2-11}
        &  \multirow{4}{*}{$ 10^{-3}$}  & $1$ & 1.211e-1 & 6.689e-1 & 3.955e+0  &3.080e+1  & 7.411e-2 & 3.685e-1 &  2.076e+0  &  1.544e+1 \\
       & & $10^{2}$  &1.376e-1  &  7.464e-1 & 4.985e+0 & 3.863e+1 & 7.857e-2& 4.288e-1  &  2.481e+0& 1.960e+1 \\
       & & $10^{4}$  &  1.632e-1 & 9.048e-1  & 5.656e+0 &  4.642e+1 &9.875e-2  & 5.343e-1&2.927e+0   & 2.256e+1\\
        & & $10^{6}$& 1.001e-1  & 5.951e-1 &3.851e+0   &  3.001e+1& 1.523e-1 & 7.238e-1& 4.397e+0 & 3.189e+1\\
       \cline{2-11}
        &  \multirow{4}{*}{$10^{-2}$} & $1$ & 1.170e-1 &  6.766e-1  & 4.009e+0 &2.940e+1  & 7.007e-2&3.701e-1  & 2.133e+0  & 1.624e+1\\
       & & $10^{2}$  & 1.440e-1 & 7.843e-1  & 5.095e+0 & 3.787e+1 & 8.214e-2& 4.552e-1 &  2.690e+0 &1.892e+1 \\
       & & $10^{4}$  & 1.696e-1  & 9.508e-1 & 5.795e+0  & 4.513e+1  &1.052e-1 & 5.663e-1 & 3.227e+0 &2.246e+1  \\
        & & $10^{6}$ & 1.092e-1  & 6.224e-1 &  3.917e+0  &  2.930e+1  &1.554e-1  &7.729e-1 & 4.454e+0 &  3.262e+1 \\
        \hline
    \end{tabular}
    \label{Poro-3field-2D-noreg-time}
\end{table}

\begin{table}[htb!]
    \centering
    \caption{The 2D example for linear poroelasticity with regularization ($\rho = 0.01 \lambda_{\min} (M_p^{\circ})$): The number of MINRES and GMRES iterations required to reach convergence for preconditioned systems $\mathcal{P}_{d}^{-1} \mathcal{A}_3$ and $\mathcal{P}_{t}^{-1} \mathcal{A}_3$, respectively, with various values of $\lambda$, $\kappa$, and $\Delta t $.}
        \begin{tabular}{|c|c|c|c|c|c|c||c|c|c|c|c|}
        \hline
       & & & \multicolumn{8}{c|}{$N$} \\ \cline{4-11}
       $\kappa$ & $\Delta t$ & $\lambda$ & 918 & 3680 & 14728 & 58608 & 918 & 3680 & 14728 & 58608 \\ \hline
       & & & \multicolumn{4}{c||}{MINRES} & \multicolumn{4}{c|}{GMRES}\\ \hline
       \multirow{12}{*}{$1$}   &  \multirow{4}{*}{$10^{-6}$}  & $1$ & 44 & 46 &  46& 48 & 24& 23& 23 &23 \\ 
       & & $10^{2}$  & 52 & 54 & 58 & 60 & 27 & 27 &  28&29 \\ 
       & & $10^{4}$  & 56 &56  & 60 &  63& 28& 28& 29 & 30\\ 
        & & $10^{6}$ &  56& 56 & 60 & 63 & 28& 28& 29 & 30\\
       \cline{2-11}
        &  \multirow{4}{*}{$ 10^{-3}$}  & $1$ & 41 & 43 & 43 &  45& 22&22 & 21 & 21\\ 
       & & $10^{2}$  & 52 &54  & 58 &  60& 27&27 & 28 & 29\\
       & & $10^{4}$  & 56 & 56 & 60 & 63 &28 & 28&  29& 30\\
        & & $10^{6}$ & 56 & 56 &  60&  63&28 & 28&  29& 30\\
       \cline{2-11}
        &  \multirow{4}{*}{$ 10^{-2}$}  & $1$ & 38 & 38 &  38&  40&19 &19&  19&19 \\
       & & $10^{2}$  & 52 &54  &  58&  60&27 &27 & 28 &29 \\
       & & $10^{4}$  &  56& 56 &  60& 63 & 28& 28& 29 & 30\\ 
        & & $10^{6}$ & 56 & 56 & 60 & 63 & 28& 28& 29 & 30\\
        \hline
       \multirow{12}{*}{$10^{-3}$}  &  \multirow{4}{*}{$ 10^{-6}$}  & $1$ & 44 & 46 & 46 & 46 & 24& 23& 23 & 23\\ 
       & & $10^{2}$  & 52 & 54 & 58 & 60 &27 & 27& 28 & 29\\
       & & $10^{4}$ & 56 & 56 & 60 &  63& 28& 28&  29& 30\\
        & & $10^{6}$ &  56& 56 & 60 & 63 &28& 28&  29& 30\\
       \cline{2-11}
        &  \multirow{4}{*}{$ 10^{-3}$}  & $1$ & 44 & 46 & 46 &  48& 24& 23& 23 & 23\\
       & & $10^{2}$  &  52& 54 & 58 & 60 & 27& 27& 28 & 29\\
       & & $10^{4}$ & 56 & 56 & 60 & 63 &28 &28 & 29 & 30\\
        & & $10^{6}$ & 56 & 56 & 60 & 63 &28 &28 & 29 & 30\\
       \cline{2-11}
        &  \multirow{4}{*}{$ 10^{-2}$}  & $1$ & 44 & 46 & 46 &  48& 23& 23& 23 &23 \\
       & & $10^{2}$  & 52 &54  & 58 & 60 & 27&27 & 28 & 29\\ 
       & & $10^{4}$ & 56 & 56 & 60 & 63 &28 &28 & 29 & 30\\
        & & $10^{6}$ & 56 & 56 & 60 & 63 & 28 &28 & 29 & 30\\
        \hline
       \multirow{12}{*}{$10^{-6}$}  &  \multirow{4}{*}{$10^{-6}$}  & 1 &44 &46  &46  & 46 & 24& 23& 23 & 23\\
       & & $10^{2}$  &  52& 54 & 58 & 60 & 27& 27& 28 & 29\\
       & & $10^{4}$  & 56 & 56 &60  & 63 & 28& 28& 29 & 30\\
        & & $10^{6}$ & 56 & 56 &60  & 63 & 28& 28& 29 & 30\\
       \cline{2-11}
        &  \multirow{4}{*}{$ 10^{-3}$}  & $1$  &44 &46  &46  & 46 & 24& 23& 23 & 23\\
       & & $10^{2}$  & 52 & 54 & 58 & 60 & 27& 27& 28 & 29\\
       & & $10^{4}$  & 56 &56  & 60 & 63 & 28& 28& 29 & 30\\
        & & $10^{6}$& 56 &56  & 60 & 63 & 28& 28& 29 & 30\\
       \cline{2-11}
        &  \multirow{4}{*}{$10^{-2}$} & $1$ & 44 &46  & 46 & 46 &  24& 23& 23 & 23\\
       & & $10^{2}$  & 52 & 54 &  58& 60 & 27& 27& 28 & 29\\
       & & $10^{4}$  & 56 & 56 & 60 &  63& 28& 28& 29 & 30\\
        & & $10^{6}$ & 56 & 56 & 60 &  63& 28& 28& 29 & 30\\
        \hline
    \end{tabular}
    \label{Poro-3field-2D-rho001}
\end{table}

\begin{table}[htb!]
    \centering
    \caption{The 2D example for linear poroelasticity with regularization ($\rho = \lambda_{\min} (M_p^{\circ})$): The number of MINRES and GMRES iterations required to reach convergence for preconditioned systems $\mathcal{P}_{d}^{-1} \mathcal{A}_3$ and $\mathcal{P}_{t}^{-1} \mathcal{A}_3$, respectively, with various values of $\lambda$, $\kappa$, and $\Delta t $.}
        \begin{tabular}{|c|c|c|c|c|c|c||c|c|c|c|c|}
        \hline
       & & & \multicolumn{8}{c|}{$N$} \\ \cline{4-11}
       $\kappa$ & $\Delta t$ & $\lambda$ & 918 & 3680 & 14728 & 58608 & 918 & 3680 & 14728 & 58608 \\ \hline
       & & & \multicolumn{4}{c||}{MINRES} & \multicolumn{4}{c|}{GMRES}\\ \hline
       \multirow{12}{*}{$1$}   &  \multirow{4}{*}{$10^{-6}$}  & $1$ & 44 & 46 &  48& 49 & 24& 23& 23 &23 \\ 
       & & $10^{2}$  & 43 & 46 & 51 & 53 & 22 & 23 &  24&25 \\ 
       & & $10^{4}$  & 44 &47  & 51 &  53& 22 & 23 &  24&25\\ 
        & & $10^{6}$ &  44& 47 & 51 & 53 & 22 & 23 &  24&25\\
       \cline{2-11}
        &  \multirow{4}{*}{$ 10^{-3}$}  & $1$ & 42 & 43 & 44 &  45& 22&22 & 21 & 21\\ 
       & & $10^{2}$  & 43 &46  & 49 &  53& 22&23 & 23 & 24\\
       & & $10^{4}$  & 43 & 46 & 51 & 53 &21 & 23&  24& 24\\
        & & $10^{6}$ & 43 & 47 &  51&  53&21 & 23&  24& 24\\
       \cline{2-11}
        &  \multirow{4}{*}{$ 10^{-2}$}  & $1$ & 36 & 38 &  40&  40&19 &19&  19&18 \\
       & & $10^{2}$  & 43 &46  &  48&  50&22 &23 & 23 &24 \\
       & & $10^{4}$  &  43& 46 &  51& 53 & 22& 23& 24 & 24\\ 
        & & $10^{6}$ & 43 & 46 & 51 & 52 & 22& 23& 24 & 24\\
        \hline
       \multirow{12}{*}{$10^{-3}$}  &  \multirow{4}{*}{$ 10^{-6}$}  & $1$ & 44 & 46 & 46 & 48 & 24& 23& 23 & 23\\ 
       & & $10^{2}$  & 43 & 47 & 51 & 53 &22 & 23& 24 & 25\\
       & & $10^{4}$ & 44 & 47 & 51 &  53&22 & 23& 24 & 25\\
        & & $10^{6}$ &  44& 46 & 51 & 53 &22 & 23& 24 & 25\\
       \cline{2-11}
        &  \multirow{4}{*}{$ 10^{-3}$}  & $1$ & 44 & 46 & 48 &  49& 24& 23& 23 & 23\\
       & & $10^{2}$  &  43& 46 & 51 & 53 &22 & 23& 24 & 25\\
       & & $10^{4}$ & 44 & 46 & 51 & 53 &22 & 23& 24 & 25\\
        & & $10^{6}$ & 44 & 47 & 51 & 53 &22 & 23& 24 & 25\\
       \cline{2-11}
        &  \multirow{4}{*}{$ 10^{-2}$}  & $1$ & 44 & 46 & 47 &  48& 24& 23& 23 &23 \\
       & & $10^{2}$  & 43 &46  & 51 & 53 &22 & 23& 24 & 25\\ 
       & & $10^{4}$ & 44 & 46 & 51 & 53 &22 & 23& 24 & 25\\
        & & $10^{6}$ & 44 & 47 & 51 & 53 &22 & 23& 24 & 25\\
        \hline
       \multirow{12}{*}{$10^{-6}$}  &  \multirow{4}{*}{$10^{-6}$}  & 1 &44 &46  &46  & 48 & 24& 23& 23 & 23\\
       & & $10^{2}$  &  43& 47 & 51 & 53 &22 & 23& 24 & 25\\
       & & $10^{4}$  & 44 & 46 &51  & 53 &22 & 23& 24 & 25\\
        & & $10^{6}$ & 44 & 46 &51  & 53 &22 & 23& 24 & 25\\
       \cline{2-11}
        &  \multirow{4}{*}{$ 10^{-3}$}  & $1$  &44 &46  &46  & 48 & 24& 23& 23 & 23\\
       & & $10^{2}$  & 43 & 47 & 51 & 53 &22 & 23& 24 & 25\\
       & & $10^{4}$  & 44 &46  & 51 & 53 &22 & 23& 24 & 25\\
        & & $10^{6}$& 44 &47  & 51 & 53 &22 & 23& 24 & 25\\
       \cline{2-11}
        &  \multirow{4}{*}{$10^{-2}$} & $1$ & 44 &46  & 46 & 46 &  24& 23& 23 & 23\\
       & & $10^{2}$  & 43 & 47 &  51& 53 &22 & 23& 24 & 25\\
       & & $10^{4}$  & 45 & 47 & 51 &  53&22 & 23& 24 & 25\\
        & & $10^{6}$ & 45 & 47 & 51 &  53&22 & 23& 24 & 25\\
        \hline
    \end{tabular}
    \label{Poro-3field-2D-rho1}
\end{table}

\begin{table}[!htb]
\centering
\caption{The 2D example for linear poroelasticity with regularization ($\rho = 0.01 \lambda_{\min} (M_p^{\circ})$): The number of MINRES iterations with various relative tolerances for inner PCG  for the 2D example with $\lambda = 10^6$, $\kappa = 10^{-6}$, and $\Delta t = 10^{-6}$. The relative tolerance for outer iteration (MINRES) is $10^{-8}$.}
\begin{tabular}{|c|c|c|c|c|}
\hline
\diagbox{$\text{relTol}_{PCG}$}{$N$} & $918$ & $3680$ & $14728$ & $58608$ \\
\hline
$10^{-2}$ &62 & 76& 103&89 \\
$10^{-3}$ &61 & 63& 75&76 \\
$10^{-6}$ &48 &52 &56 & 58\\
$10^{-8}$ &48 &52 &56 & 58\\
$10^{-10}$ &48 &52 &56 & 58\\
$10^{-12}$ &48 &52 &56 & 58\\
\hline
\end{tabular}
\label{table-block-inverse}
\end{table}

\subsection{Varying $\kappa$ with jumps}
In this example, we consider
\[
\kappa =
\begin{cases}
\kappa_1, & 0.25 \le x \le 0.75, \\
1,        & \text{elsewhere},
\end{cases}
\]
with $\kappa_1 \in \{10^{-10}, 10^{-6}, 10^{-2}\}.$
The remaining parameters are chosen as
\[
\Delta t = 10^{-6},\; 
\lambda \in \{1,\,10^{2},\,10^{4},\,10^{6}\},\;
\alpha \in \{1,\,10^{-2},\,10^{-4}\}.
\]
We take $\V{u}$, $p$, and $\V{f}$ as in Section~\ref{ex-smooth} and compute the right-hand side $s$ accordingly.

It is worth noting that jumps in $\kappa$ may occur within elements due to
the use of triangular meshes. Moreover, the convergence analysis for MINRES and GMRES
presented in the previous section applies only to the case of constant $\kappa$. 
Nevertheless, the block preconditioners remain applicable from an algorithmic standpoint
for variable and possibly discontinuous $\kappa$. This example is used to examine
whether these preconditioners continue to perform well under such conditions.

Table~\ref{tab-nonconstK-gmres} presents the iteration counts of MINRES and GMRES
for the preconditioned system. As discussed in Section~\ref{sec:poro3-analysis}
and Lemma~\ref{lem:eigen_bound_3field}, the performance of MINRES and GMRES
is essentially independent of the problem parameters in the case of constant $\kappa$. 
This behavior is also confirmed by the results in Table~\ref{tab-nonconstK-gmres}
for the example with jumps in the permeability.
Indeed, the iteration counts of both MINRES and GMRES remain nearly constant as $\kappa_1$ varies
when the other parameters are fixed, and they also remain stable under mesh refinement
for fixed parameter values. For a fixed $\kappa_1$, variations in $\lambda$ and $\alpha$
lead only to modest changes in the iteration counts. For example, on a mesh with 58{,}608 elements,
when $\alpha = 10^{-2}$ and $\lambda = 10^{6}$, MINRES requires 52 iterations to converge,
whereas for $\alpha = 10^{-4}$ and $\lambda = 1$, convergence is achieved in 26 iterations. 
Overall, these results indicate that variations in $\alpha$ result only in moderate
and acceptable changes in the iteration counts.

\begin{table}[htb!]
\caption{The 2D example for linear poroelasticity with nonconstant permeability in three-field formulation with regularization ($\rho = 0.1 \lambda_{\min} (M_p^{\circ})$): The number of MINRES and GMRES iterations required to reach convergence for preconditioned systems $\mathcal{P}_{d}^{-1} \mathcal{A}_3$ and $\mathcal{P}_{t}^{-1} \mathcal{A}_3$, respectively, with various $\lambda$, $\alpha$, and $\Delta t = 10^{-6}$.}
\centering
\begin{tabular}{|c|c|c|c|c|c||c|c|c|}
\hline
$N$ & $\alpha$ & $\lambda$ & \multicolumn{3}{c||}{MINRES} & \multicolumn{3}{c|}{GMRES}\\
\cline{4-9}
    &          &         & $\kappa_1 =10^{-2}$ & $\kappa_1 = 10^{-6}$ & $\kappa_1 = 10^{-10}$  & $\kappa_1 =10^{-2}$ & $\kappa_1 = 10^{-6}$ & $\kappa_1 = 10^{-10}$ \\
\hline
% \multirow{9}{*}{$918$}
%  & \multirow{4}{*}{$1$}
%  & $1$  &44 & 44&    44   & 24 & 24 & 24 \\
%  &        & $10^{2}$   & & &       &  &  &  \\
%  &        & $10^{4}$   & 48&48 &  48     & 25 &25  & 25 \\
%  &         & $10^{6}$  & 48& 48&    48   & 25 &25  & 25 \\
% \cline{2-9}
%  & \multirow{4}{*}{$10^{-2}$}
%  & $1$   & 31& 31&    31    & 16 &  16& 16 \\
%   &        & $10^{2}$   & & &       &  &  &  \\
%  &           & $10^{4}$  & 46& 46&   46  & 24 & 24 & 24 \\
%  &            & $10^{6}$  & 46&46 &  46  & 24 & 24 & 24 \\
% \cline{2-9}
%  & \multirow{4}{*}{$10^{-4}$}
%  & $1$    & 25& 25&   25    & 15 & 15 & 15 \\
%   &        & $10^{2}$   & & &       &  &  &  \\
%  &            & $10^{4}$  &34 &34 &34    & 24 & 24 & 24 \\
%  &             & $10^{6}$  & 34&34 & 34  & 24 &24  & 24 \\
% \hline
\multirow{9}{*}{$3680$}
 & \multirow{4}{*}{$1$}
 & $1$   & 44&44 &   44     & 23 & 23 & 23 \\
  &        & $10^{2}$   & 51& 51&     51  & 25 & 25 &  25\\
 &          & $10^{4}$  & 52& 52&   52   & 26 & 26 &26  \\
 &            & $10^{6}$   & 52&52 &  52 & 26 & 26 &  26\\
\cline{2-9}
 & \multirow{4}{*}{$10^{-2}$}
 & $1$    &32 & 32&    32   & 16 &  16&16  \\
  &        & $10^{2}$   & 49& 49& 49  & 25 &  25&  25\\
 &          & $10^{4}$   & 50&50 &  50   & 25 & 25 & 25 \\
 &           & $10^{6}$  & 50&50 &   50  & 25 &25  & 25 \\
\cline{2-9}
 & \multirow{4}{*}{$10^{-4}$}
 & $1$   & 26& 26&     26    & 15 & 15 & 15 \\
  &        & $10^{2}$   & 32& 32& 32 & 24 & 24 & 24 \\
 &            & $10^{4}$  &36 &36 & 36    & 25 & 25 & 25 \\
 &           & $10^{6}$  &36 &36 & 36  &  25 & 25 & 25  \\
\hline
\multirow{9}{*}{$14728$}
 & \multirow{4}{*}{$1$}
 & $1$   & 46& 46&     46   & 23 & 23 & 23 \\
  &        & $10^{2}$   & 53& 53&     53  & 26 &26  & 26 \\
 &             & $10^{4}$ &56 & 56& 56   &  27&  27& 27 \\
 &            & $10^{6}$   &56 &56 & 56  &  27&  27& 27 \\
\cline{2-9}
 & \multirow{4}{*}{$10^{-2}$}
 & $1$   & 33&33 &    33    &  16&  16& 16 \\
  &  & $10^{2}$   & 48 &48 & 48 & 24 & 24 & 24 \\
 &           & $10^{4}$ & 52& 52&  52    & 25 &  25& 25 \\
 &           & $10^{6}$   &52 & 52&  52  & 25 & 25 &  25\\
\cline{2-9}
 & \multirow{4}{*}{$10^{-4}$}
 & $1$   & 26& 26&    26    & 15 & 15 &  15\\
  &        & $10^{2}$   &32  &32 &  32 & 23 & 23 & 23 \\
 &           & $10^{4}$  & 36&36 &  36   & 24 & 24 &24  \\
 &            & $10^{6}$ &36 & 36&  36   & 24 &24  &  24\\
 \cline{2-9}
\hline
\multirow{9}{*}{$58608$}
 & \multirow{4}{*}{$1$}
 & $1$  &48 &48 &       48  & 23 & 23 &23  \\
  &        & $10^{2}$   & 54& 54&     54 & 26 &  26&  26\\
 &         & $10^{4}$  & 58& 58&   58    & 28 & 28 &28  \\
 &          & $10^{6}$  & 58& 58&   58     & 28 & 28 &28  \\
\cline{2-9}
 & \multirow{4}{*}{$10^{-2}$}
 & $1$  & 34& 34&   34      & 16 & 16 & 16 \\
  &        & $10^{2}$   & 50& 50 & 50 & 24 & 24 & 24 \\
 &            & $10^{4}$   &52 &52 &52   &  25& 25 & 25 \\
 &        & $10^{6}$     &52 &52 &52   &25  & 25 & 25 \\
\cline{2-9}
 & \multirow{4}{*}{$10^{-4}$}
 & $1$  & 26& 26&    26     &15  & 15 & 15 \\
  &   & $10^{2}$   &32 & 32 & 32 & 24 & 24 &  24\\
 &             & $10^{4}$  & 32& 32&32  &25  & 25 & 25  \\
 &            & $10^{6}$ & 32& 32&32  &25  & 25 & 25  \\
\hline
\end{tabular}
\label{tab-nonconstK-gmres}
\end{table}

%%%%%%%%%%%%%%%%%%%%%%%%%%%%%%%%%%%%%%%%%%%%%%%%%%%%%%%%%%%%%%%
\section{Conclusions}
\label{SEC:conclusions}

In the preceding sections, we have investigated an inherent regularization strategy and block Schur complement
preconditioning for the efficient iterative solution of linear poroelasticity problems discretized
using the lowest-order weak Galerkin finite element method in space and the implicit Euler scheme in time.
In the locking regime, where $\epsilon = \mu/(\lambda + \mu) \to 0$, the leading block
$\epsilon A_1 + A_0$ becomes nearly singular, rendering both this block and the full system
challenging for iterative solvers.

To address this difficulty, we introduced a numerical pressure variable 
$\V{z}_h$ and reformulated the original two-field system as a three-field formulation.
For this three-field system, we proposed an inherent regularization strategy in which the term 
$-\rho \V{w}\V{w}^T \V{z}_h = \V{0}$ is added to the third block equation, where
$\V{w}$ is defined in \eqref{w-1}. The regularized system preserves the original solution,
since the condition $\V{w}^T \V{z}_h = 0$ (equivalently, $-\rho \V{w}\V{w}^T \V{z}_h = \V{0}$)
is an inherent identity of the original formulation; see \eqref{3field_eqn2} and \eqref{3field_eqn4}.
The resulting Schur complement has eigenvalues that are uniformly bounded above and below by positive constants,
staying away from zero (cf. Lemma~\ref{lem:eigen_bound_3field}). Consequently, MINRES and GMRES
equipped with inexact block diagonal or triangular Schur complement preconditioners exhibit convergence
rates that are essentially independent of both the mesh size and the locking parameter.

Moreover, as discussed in Section~\ref{sec:poro3-implementation}, the proposed solver requires
only two levels of nested iteration: an outer iteration using MINRES or GMRES for the full system,
and an inner iteration using PCG for the block $A_1$. This structure leads to improved computational
efficiency while maintaining robustness. Although the two-field formulation can be solved by
reformulating it as a nested saddle point system in the locking regime,
the resulting solver would involve three levels of nested iteration (outer iterations for the full system,
inner iterations for the leading block, and additional inner iterations for $A_1$).
In comparison, the three-field formulation considered here is significantly more efficient.

Numerical results for linear poroelasticity problems in both two and three dimensions were presented
in Section~\ref{SEC:numerical}. These results confirm the effectiveness of the proposed regularization
strategy and the robustness of the block preconditioners with respect to both the mesh size and
the locking parameter.

We remark that the regularization strategy considered in this work can also be applied
to the two-field formulation \eqref{2by2Scheme_matrix2}. In this case, the implementation
of block diagonal or triangular Schur complement preconditioning requires the inversion of
the nearly singular leading block. This inversion can be carried out by reformulating
the leading block as a saddle point system, applying the inherent regularization strategy,
and solving the resulting system using preconditioned MINRES or GMRES. A similar convergence
analysis can then be performed. In addition, the regularization strategy can be extended
to the iterative solution of linear elasticity problems.

Finally, we note that, throughout this work, we employ the lowest-order weak Galerkin discretization, and the parameters $\lambda$, $\mu$, and $\kappa$ are assumed to be constant, except for one example in which $\kappa$ exhibits a discontinuity. An interesting direction for future research is to investigate  whether the analysis developed here can be extended to higher-order weak Galerkin discretizations and to settings with spatially varying parameters.

%%%%%%%%%%%%%
\section*{Acknowledgments}

W.~Huang was supported in part by the Simons Foundation grant MPS-TSM-00002397.
The authors thank the anonymous referees for their insightful comments and suggestions, which have improved the quality of this paper.

\bibliographystyle{abbrv}
% \bibliography{bibfile}

\begin{thebibliography}{10}

\bibitem{Adler6_SISC_2020}
J.~H. Adler, F.~J. Gaspar, X.~Hu, P.~Ohm, C.~Rodrigo, and L.~T. Zikatanov.
\newblock Robust preconditioners for a new stabilized discretization of the
  poroelastic equations.
\newblock {\em SIAM J. Sci. Comput.}, 42:B761--B791, 2020.

\bibitem{AmbartKhatYotov_CMAME_2020}
I.~Ambartsumyan, E.~Khattatov, and I.~Yotov.
\newblock A coupled multipoint stress - multipoint flux mixed finite element
  method for the {Biot} system of poroelasticity.
\newblock {\em Comput. Methods Appl. Mech. Engrg.}, 372:113407, 2020.

\bibitem{BenziGolubLiesen-2005}
M.~Benzi, G.~H. Golub, and J.~Liesen.
\newblock Numerical solution of saddle point problems.
\newblock {\em Acta Numer.}, 14:1--137, 2005.

\bibitem{Benzi2008}
M.~Benzi and A.~J. Wathen.
\newblock {\em Some Preconditioning Techniques for Saddle Point Problems},
  pages 195--211.
\newblock Springer Berlin Heidelberg, Berlin, Heidelberg, 2008.

\bibitem{Boffi2008}
D.~Boffi, F.~Brezzi, and M.~Fortin.
\newblock {\em Finite Elements for the {Stokes} Problem}, pages 45--100.
\newblock Springer Berlin Heidelberg, Berlin, Heidelberg, 2008.

\bibitem{Boon4_SISC_2021}
W.~M. Boon, M.~Kuchta, K.-A. Mardal, and R.~Ruiz-Baier.
\newblock Robust preconditioners for perturbed saddle point problems and
  conservative discretizations of {Biot}'s equations utilizing total pressure.
\newblock {\em SIAM J. Sci. Comput.}, 43:B961--B983, 2021.

\bibitem{Burger4_AdvComputMath_2021}
R.~B\"{u}rger, S.~Kumar, D.~Mora, R.~Ruiz-Baier, and N.~Verma.
\newblock Virtual element methods for the three-field formulation of
  time-dependent linear poroelasticity.
\newblock {\em Adv. Comput. Math.}, 47:Article\# 2, 2021.

\bibitem{ChenHongXuYang_CMAME_2020}
S.~Chen, Q.~Hong, J.~Xu, and K.~Yang.
\newblock Robust block preconditioners for poroelasticity.
\newblock {\em Comput. Methods Appl. Mech. Engrg.}, 369:113229, 2020.

\bibitem{Coulet4_ComputGeosci_2020}
J.~Coulet, I.~Faille, V.~Girault, N.~Guy, and F.~Nataf.
\newblock A fully coupled scheme using virtual element method and finite volume
  for poroelasticity.
\newblock {\em Comput. Geosci.}, 24:381--403, 2020.

\bibitem{Elman-2014}
H.~C. Elman, D.~J. Silvester, and A.~J. Wathen.
\newblock {\em Finite elements and fast iterative solvers: with applications in
  incompressible fluid dynamics}.
\newblock Numerical Mathematics and Scientific Computation. Oxford University
  Press, Oxford, second edition, 2014.

\bibitem{Greenbaum-1997}
A.~Greenbaum.
\newblock {\em Iterative methods for solving linear systems}, volume~17 of {\em
  Frontiers in Applied Mathematics}.
\newblock Society for Industrial and Applied Mathematics (SIAM), Philadelphia,
  PA, 1997.

\bibitem{He-Jing-Feng-2025}
L.~He, L.~Jing, and M.~Feng.
\newblock New stabilized mixed finite element methods for two-field
  poroelasticity with low permeability.
\newblock {\em Appl. Math. Comput.}, 494:Paper No. 129285, 26, 2025.

\bibitem{Hong-2023-MathComp}
Q.~Hong, J.~Kraus, M.~Lymbery, and F.~Philo.
\newblock A new practical framework for the stability analysis of perturbed
  saddle point problems and applications.
\newblock {\em Math. Comp.}, 92:607--634, 2023.

\bibitem{HuangWang_CiCP_2015}
W.~Huang and Y.~Wang.
\newblock Discrete maximum principle for the weak {G}alerkin method for
  anisotropic diffusion problems.
\newblock {\em Comm. Comput. Phys.}, 18:65–90, 2015.

\bibitem{Huang-Wang-2025-Stokes-reg}
W.~Huang and Z.~Wang.
\newblock A general regularization strategy and convergence analysis for the
  iterative solution of singular {S}tokes problems.
\newblock {\em Commun. Nonlinear Sci. Numer. Simul.}, 161: Article No. 110059, 2026.


\bibitem{HuangWang_CiCP_2025}
W.~Huang and Z.~Wang.
\newblock Convergence analysis of iterative solution with inexact block
  preconditioning for weak {Galerkin} finite element approximation of {Stokes}
  flow.
\newblock {\em Commun. Comput. Phys.}, 40:479--504, 2026.

\bibitem{HuangWang_2025}
W.~Huang and Z.~Wang.
\newblock Efficient parameter-robust preconditioners for linear poroelasticity and elasticity in the primal formulation.
\newblock {\em SIAM J. Sci. Comput.}, 48(3):A1536--A1563, 2026.



\bibitem{HuangWang_2025_CGWG}
W.~Huang and Z.~Wang.
\newblock Parameter-free inexact block Schur complement preconditioning for linear poroelasticity under a hybrid Bernardi–Raugel and weak Galerkin finite element discretization.
\newblock {\em AMS Contemp. Math.}, 2026. (To appear).


\bibitem{KADEETHUM2021110030}
T.~Kadeethum, H.~Nick, S.~Lee, and F.~Ballarin.
\newblock Enriched {G}alerkin discretization for modeling poroelasticity and
  permeability alteration in heterogeneous porous media.
\newblock {\em J. Comput. Phys.}, 427:110030, 2021.

\bibitem{KamenskiHuangXu_2014}
L.~Kamenski, W.~Huang, and H.~Xu.
\newblock Conditioning of finite element equations with arbitrary
anisotropic meshes.
\newblock {\em Math. Comp.}, 83:2187--2211, 2014.

\bibitem{Lee-SISC-2017}
J.~J. Lee, K.-A. Mardal, and R.~Winther.
\newblock Parameter-robust discretization and preconditioning of {B}iot's
  consolidation model.
\newblock {\em SIAM J. Sci. Comput.}, 39:A1--A24, 2017.

\bibitem{SISCLeePierMarRog2019}
J.~J. Lee, E.~Piersanti, K.-A. Mardal, and M.~E. Rognes.
\newblock A mixed finite element method for nearly incompressible
  multiple-network poroelasticity.
\newblock {\em SIAM J. Sci. Comput.}, 41:A722--A747, 2019.

\bibitem{LeeYi_JSC_2023}
S.~Lee and S.-Y. Yi.
\newblock Locking-free and locally-conservative enriched {Galerkin} method for
  poroelasticity.
\newblock {\em J. Sci. Comput.}, 94:26, 2023.

\bibitem{Luber-2024}
T.~Luber and S.~Sysala.
\newblock Robust block diagonal preconditioners for poroelastic problems with
  strongly heterogeneous material.
\newblock {\em Numer. Linear Algebra Appl.}, 31:Paper No. e2546, 18, 2024.

\bibitem{PestanaWathen-2015}
J.~Pestana and A.~J. Wathen.
\newblock Natural preconditioning and iterative methods for saddle point
  systems.
\newblock {\em SIAM Rev.}, 57(1):71--91, 2015.

\bibitem{RIVIERE2017666}
B.~Rivière, J.~Tan, and T.~Thompson.
\newblock Error analysis of primal discontinuous {G}alerkin methods for a mixed
  formulation of the {B}iot equations.
\newblock {\em Comput. Math. Appl.}, 73:666--683, 2017.

\bibitem{Rodrigo6_SeMA_2024}
C.~Rodrigo, F.~J. Gaspar, J.~Adler, X.~Hu, P.~Ohm, and L.~Zikatanov.
\newblock Parameter-robust preconditioners for {B}iot's model.
\newblock {\em SeMA J.}, 81:51--80, 2024.

\bibitem{SilvesterWathen_SINUM_1994}
D.~Silvester and A.~Wathen.
\newblock Fast iterative solution of stabilised {S}tokes systems. {II}. {U}sing
  general block preconditioners.
\newblock {\em SIAM J. Numer. Anal.}, 31:1352--1367, 1994.

\bibitem{WANGYe2013103}
J.~Wang and X.~Ye.
\newblock A weak {Galerkin} finite element method for second-order elliptic
  problems.
\newblock {\em J. Comput. Appl. Math.}, 241:103--115, 2013.

\bibitem{Wang2TavLiu_JCAM_2024}
R.~Wang, Z.~Wang, S.~Tavener, and J.~Liu.
\newblock Full weak {G}alerkin finite element discretizations for
  poroelasticity problems in the primal formulation.
\newblock {\em J. Comput. Appl. Math.}, 443:115754, 2024.

\bibitem{WathenSilvester_SINUM_1993}
A.~Wathen and D.~Silvester.
\newblock Fast iterative solution of stabilised {S}tokes systems. {I}. {U}sing
  simple diagonal preconditioners.
\newblock {\em SIAM J. Numer. Anal.}, 30:630--649, 1993.

\bibitem{Yi_SISC_2017}
S.-Y. Yi.
\newblock A study of two modes of locking in poroelasticity.
\newblock {\em SIAM J. Numer. Anal.}, 55:1915--1936, 2017.

\bibitem{ZhaoChenCaiSun}
J.~Zhao, H.~Chen, M.~Cai, and S.~Sun.
\newblock An optimally convergent parallel splitting algorithm for the
multiple-network poroelasticity model.
\newblock {\em J. Comput. Phys.}, 539:Paper No. 114214, 22, 2025.



\end{thebibliography}

\end{document}